\documentclass[11pt]{article}
\usepackage{setspace}
\singlespace
\usepackage{amsmath,amsthm,amssymb,color,array,multirow}
\usepackage{epsfig,amsfonts,latexsym, hyperref,  mathrsfs}
\usepackage{natbib}
\usepackage{geometry}
\usepackage{tikz}
\usetikzlibrary{shadows,trees}

\hypersetup{
    colorlinks=true,
    linkcolor=red,
    citecolor=blue,
    urlcolor=blue,
    pdfborder={0 0 0}
}

\newtheorem{remark}{Remark}[section]
\newtheorem{theorem}{Theorem}[section]

\theoremstyle{definition}
\newtheorem{definition}{Definition}[section]
\newtheorem{Example}{Example}[section]
\newcommand{\one}{{\bf 1}} 

\numberwithin{equation}{section}

\allowdisplaybreaks

\begin{document}
\def\shorttitle{Cross-covariance matrices}
 \def\shortauthors{A. Dey}
\title{\textbf{\Large \sc
%When is a variable free Gaussian?
%Independence of $\bar X$ and $s^2$ in the free world
Joint convergence of sample cross-covariance matrices
}}
 \author{
%\hspace{0.05\textwidth}
 \parbox[t]{0.3\textwidth}{{Monika Bhattacharjee}
 \thanks{Research  supported by seed grant from IIT Bombay and Start-up Research Grant from SERB, Dept.~of Science and Technology, Govt.~of India.}
 \\ {\small Dept. of Math. \\
Indian Institute of Technolgy\\  
monaiidexp.gamma@gmail.com\\ }}
\hspace{0.01\textwidth}
 \parbox[t]{0.3\textwidth}{{Arup Bose}
  \thanks{Research  supported by J.C.~Bose National Fellowship, Dept.~of Science and Technology, Govt.~of India.}
 \\ {\small Stat.-Math Unit\\
Indian Statistical Institute\\  
%203 B.T. Road\\
%Kolkata 700108\\
%INDIA \\ 
bosearu@gmail.com\\ }}
\hspace{0.01\textwidth}
 \parbox[t]{0.20\textwidth}{{{Apratim Dey}
 \\ {\small Dept. of Stat.\\
Stanford University\\  
%203 B.T. Road\\
%Kolkata 700108\\
%INDIA \\
apd1995@stanford.edu\\}}
}
}
\maketitle

\begin{abstract}
Suppose $X$ and $Y$ are $p\times n$ matrices each with  mean $0$, variance $1$ and where all moments of any order are  uniformly bounded as $p,n \to \infty$. Moreover, the entries $(X_{ij}, Y_{ij})$ are independent across $i,j$ with a common correlation $\rho$. Let $C=n^{-1}XY^*$ be the sample cross-covariance matrix. We show that if $n, p\to \infty, p/n\to y\neq 0$, then $C$ converges in the algebraic sense and the limit moments depend only on $\rho$. Independent copies of such matrices with same $p$ but different $n$, say $\{n_l\}$,  different correlations $\{\rho_l\}$,  and different non-zero $y$'s, say $\{y_l\}$  also converge jointly and are asymptotically free. When $y=0$, the matrix $\sqrt{np^{-1}}(C-\rho I_p)$ converges to an elliptic variable with parameter $\rho^2$. In particular, this elliptic variable is circular when $\rho=0$ and is semi-circular when $\rho=1$. If we take independent $C_l$, then the matrices $\{\sqrt{n_lp^{-1}}(C_l-\rho_l I_p)\}$ converge jointly and are also asymptotically free. 
As a consequence, the limiting spectral distribution of any symmetric matrix polynomial exists and has compact support. \end{abstract}

\noindent \textbf{Key words}. Sample cross-covariance matrices, joint convergence, free independence, free cumulants, semi-circular variable, circular variable, elliptic variable, cross-covariance variable, Mar\v{c}enko-Pastur law, compound free Poisson law.\vskip10pt

\noindent \textbf{AMS 2000 Subject classification}: Primary 60B20, Secondary 46L54. 

%\begin{theorem}
%Let $(X_{n1},...,X_{nk})$ be a collection of $k$ (with $k$ fixed) symmetric Gaussian Wigner random matrices with entries having mean zero and variance $\dfrac{1}{n}$. Then $$\dfrac{1}{n}ETr(X_{n\epsilon_1}...X_{n\epsilon_j})\to \varphi(s_{\epsilon_1}...s_{\epsilon_j})$$for each $(\epsilon_1,...,\epsilon_j)\in [k]^j$, for all $j\geq1$, where $s_1,...s_k$ are freely independent standard semicircle variables on a common algebra $(\mathcal A,\varphi)$.
%\end{theorem}

%So free probability helps us to understand the relationship among the limits of the random matrices and describe the limits as some non-commutative random variables on an algebra.

\section{Introduction} 
%Various types of large dimensional random matrix models have been studied in the literature. In particular, 
The large sample behaviour of the high dimensional sample covariance matrix $S=n^{-1}XX^{*}$ where $X$ is a $p\times n$ matrix with i.i.d. entries has been extensively studied. 
%both from the theoretical and applications point of view. 
%and has been widely used in many areas including high-dimensional statistics and wireless communications. 
Under suitable moment assumptions on the entries, the convergence of its spectral distribution  when $n,p\to \infty$ and $p/n\to y\neq 0$ was originally shown in \cite{marcenkopastur1967} and the limit law is now known as the Mar\v{c}enko-Pastur law. When $y=0$, the limit spectral distribution (LSD) of $\sqrt{np^{-1}}(S-I_p)$ where $I_p$ is the $p\times p$ identity matrix, is known to be
 the (standard) semi-circle law. See  \cite{BS2010}  and  \cite{bose2018} for book-level expositions of these results.  
% autocovariance matrices (\cite{BBS2014});  
The joint algebraic convergence and asymptotic freeness of independent $S$-matrices was established in \cite{CC2004}. The joint convergence of the generalized covariance matrices and the convergence of the spectral distribution of their symmetric matrix polynomials was dealt in \cite{BB2016} and  \cite{BB2017}. %(a) ??something wrong with BB references. two references are repeated?? See reference list in the pdf. (b)  ackemann ref volume ??. (c) Bai silverstein ref, delete Vol ref or give the series name%(Erratum: \cite{BB2018e}).  %???MUST INCLUDE ERRATA/CORRIGENDUM REFERENCE. 

The semi-circle law is a central probability law in free probability and orginally arose from the study of a Wigner matrix $W_n$ which is a real symmetric matrix whose entries are i.i.d. with mean $0$ and variance $1$. The limit spectral distribution of $n^{-1/2}W_n$  is the (standard) semi-circle law. Moreover, independent copies of these matrices, when all moments are finite, converge jointly in the algebraic sense and are asymptotically free. 

If in $W_n$ the $(i,j)$-th and the $(j,i)$-th entries have a common correlation $\rho$, then it is called an elliptic matrix, and in that  case the LSD of $n^{-1/2}W_n$ is the uniform law in the interior of an ellipse centered at the origin with lengths of the major and the minor axes being $\sqrt{2(1+\rho)}$ and $\sqrt{2(1-\rho)}$.  See 
\cite{nguyenrourke}. In particular if $\rho=1$ we recover the semi-circle law result for the Wigner matrix, and if $\rho=0$ then the LSD is uniformly distributed over the unit disc. It is also known that independent copies of  elliptic matrices with possibly different $\rho$ converge jointly to elliptic elements and are asymptotically free. See \cite{boseAdh}.

Motivated by the above results, we consider the following high dimensional model: $X$ and $Y$ are two $p\times n$ random matrices  where the entries $(x_{ij}, y_{ij})$, $1\leq i\leq p$, $1\leq j \leq n$  are independent bivariate random variables with means $0$, variances $1$ and correlation $\rho$. Then the matrix $C=n^{-1}XY^{*}$ is called the sample \textit{cross-covariance matrix}.  Note that if $\rho=1$, then $X=Y$ and we recover the $S$ matrix. 

We study the joint convergence of independent copies  $\{C_l\}$ of these matrices with possibly different $\{\rho_l\}$ and $\{n_l\}$. This convergence is taken to be convergence as elements of an appropriate $*$-probability space as described below. 

Consider the $*$-probability space $(\mathcal{M}_p, \varphi_p)$ where $\mathcal{M}_p$ is the set of all $p\times p$ random matrices: 
 \begin{equation}\label{eq:alg}
\mathcal{M}_p(\mathbb{C})=\{A: A=((a_{ij}))_{p\times p} \ \text{and}\ \ \ \mathbb{E}|a_{ij}|^k < \infty\ \ \text{for all}\ \ i, j, k\},
\end{equation}
and the state $\varphi_p$ is defined as    
$$\varphi_p(A)=p^{-1}\mathbb{E}[\text{Trace}(A)].$$
 Note that $\varphi_p$ is  positive  and tracial (that is, $\varphi_p(aa^{*}) \geq 0$, and $\varphi_p(ab)=\varphi_p(ba)$ for all $a, b$).

Elements $\{A_l\}$, $1\leq i \leq t$ from  $\mathcal{M}_p(\mathbb{C})$ are said to converge jointly if for every polynomial $\Pi(A_l, 1\leq l\leq t)$ in the variables $\{A_l, A_l^{*}\}$, 
$\varphi_p(\Pi(A_l, 1\leq l\leq t))=p^{-1}\mathbb{E}[\text{Trace}(\Pi(A_l, 1\leq l\leq t))]$ converges as $p\to \infty$. The limit algebra is taken to be a $*$-algebra generated by indeterminates $\{a_l\}$ and its state is defined through the above limits as
\begin{equation}\label{eq:convdef}\varphi (\Pi(a_l, 1\leq l\leq t))=\lim_{p\to \infty} p^{-1}{\mathbb{E}}[\text{Trace}(\Pi(A_l, 1\leq l\leq t))].
\end{equation}
We shall refer to this convergence as algebraic convergence.  Note that $\varphi$ is also tracial and positive.  The matrices $\{A_l\}$ are said to be free if the limit variables $\{a_l\}$ are free.  

Two different  cases arise: for each $l$, either $n_l, p\to \infty, p/n_l\to y_l, 0 < y_l < \infty$ or $p\to \infty, p/n_l\to 0$. Suppose the entries of $\{X_l, Y_l\}$ have mean $0$, variance $1$, and moments of any order are uniformly bounded. Moreover, across $l$ the matrices are independent. Then in the first case $\{C_l\}$ converge jointly in the sense of (\ref{eq:convdef}) to say $\{c_l\}$ which are asymptotically free. 
%to say $\{c_l\}$ which, and are asymptotically free. That is, the limit variables,  are free in the limit $*$-algebra with respect to the state $\varphi$ defined in (\ref{eq:convdef}). 
The moments of each $c_l$ depend only on $y_l$ and $\rho_l$ and we give a formula for their free cumulants. In the second case $\{\sqrt{n_l p^{-1}}(C_l-\rho_lI_p)\}$, $1\leq l \leq t$ converge to free elliptic variables with parameters $\rho_l^2$. The algebraic convergence results on independent $S$-matrices mentioned earlier are obtained as a special case by using $\rho=1$. 

There is another natural state on $\mathcal{M}_p(\mathbb{C})$ given by 
$$\tilde{\varphi}_p(A)=p^{-1}\text{Trace}(A).$$
Convergence with respect to this state is defined as the almost sure convergence of   
$p^{-1}\text{Trace}(\Pi$ $(A_l, 1\leq l\leq t))$ for all polynomials $\Pi$, and for our purposes 
the limits are non-random. In the present paper, it turns out that all the convergence results described above that hold for $\{C_l\}$  with respect to $\varphi_p$ also hold with respect to $\tilde{\varphi}_{p}$. 
%in (\ref{eq:convdef}) for the sequences under our consideration. 
 
A related notion of convergence for a single sequence of random matrices is the convergence of the spectral distribution. 
Suppose $A_p$ is any $p\times p$ matrix with eigenvalues $\lambda_1 , \lambda_2 , . . . , \lambda_p$. The  random probability law which puts mass $p^{-1}$ on each eigenvalue is called the empirical spectral distribution (ESD) of $A_p$.  If we take a further expectation with respect to the underlying law of the random variables, then it defines another probability law, which we shall call the expected empirical spectral distribution (EESD). If the ESD converges weakly (in probability or almost surely), the limit is called the 
limiting spectral distribution (LSD). If this limit is non-random, then it is also the limit of the EESD. Usually the convergence of the EESD is easier to establish and then the convergence of the ESD to the same limit often follows by a Borel-Cantelli type argument on the moments of the ESD.  

\cite{akemann2008} and \cite{benet2014} respectively analysed the characteristic polynomial and spectral domain,  and \cite{akemann2020non} established the weak convergence of the ESD of a cross-covariance matrix with complex Gaussian entries when  $0 < y_l <\infty$. 
Now consider any symmetric matrix polynomial $\Pi$ in the matrices $\{C_l, C_l^{*}\}$. Then the moments of the EESD of $\Pi$ are $p^{-1}
\mathbb{E}[{\rm Trace}(\Pi^k)]$. Algebraic convergence with respect to $\varphi_p$ implies that $p^{-1}\mathbb{E}[{\rm Trace}(\Pi^k)]$ converges for all integers $k\geq 1$. It is easily checked that these limiting moments define a unique probability law and hence the EESD of $\Pi$ converges weakly to this law. This convergence can be upgraded to almost sure convergence of the ESD of $\Pi$ by an estimation of the fourth  moment  of $p^{-1}{\rm Trace}(\Pi^k)-p^{-1}\mathbb{E}[{\rm Trace}(\Pi^k)]$ and applying the Borel-Cantelli Lemma. As an example, the distribution of the singular values of $C_l$ converges almost surely. Simulations suggest that the LSD exists even for non-symmetric polynomials. It is difficult to settle this rigorously for any non-symmetric polynomial, and we do not pursue this issue in this article. 
 %and the In this paper, we consider a different generalization of sample covariance matrices which is the product of two independent matrices whose $(i,j)$-th entries are correlated and this correlation is same across all entries. We refer this matrix as cross-covariance matrix. Here we establish asymptotic freeness of independent cross-covariance matrices and their marginal limits. \textbf{Need to add more ...}

\section{Necessary notions from non-commutative probability}
We briefly mention some notions of non-commutative probability that we shall need. For further details see  \cite{NS2006}.
Let $(\mathcal{A},\varphi)$ be a $\ast$-probability space.  Suppose $\{a_i: i \in I \}$ is a collection of variables from 
$\mathcal{A}$.  
Then the numbers $\{\varphi\left(\Pi(a_i,a_i^{*}:i \in I)\right): \Pi\ \text{is a finite degree monomial}\}$ are called their \textit{joint $*$-moments}.
For simplicity, we shall often write moments instead of $*-$moments.
 
Suppose $a\in \mathcal{A}$ is a \textit{self-adjoint} element, that is, $a^*=a$. Suppose there a \textit{unique} probability law $\mu_a$ on $\mathbb{R}$ such that 
\begin{equation}\label{eq:problaw}
\varphi(a^n)=\int_{\mathbb{R}} x^n\mu_{a}(dx), n =1, 2, \ldots.
\end{equation}
Then  $\mu_a$ is called the \textit{probability law of $a$}. 
If $\mu_a$ satisfies (\ref{eq:problaw}) and is compactly supported, then it is indeed unique.

Let $(\mathcal{A},\varphi)$ be a $*$-probability space. Define a sequence of \textit{multi-linear functionals}  $(\varphi_{n})_{n \in \mathbb{N}}$ on $\mathcal{A}^{n}$ via
\begin{eqnarray} \label{eqn:mmt4.1}
\varphi_{n} (a_{1},a_{2},\ldots, a_{n}):= \varphi(a_{1}a_{2}\cdots a_{n}).
\end{eqnarray}

 Extend $\{\varphi_n, n\geq 1\}$ to $\{\varphi_{\pi},  \pi \in NC(n), n \geq 1\}$ \textit{multiplicatively} 
in a recursive way by the following formula. 
If $\pi = \{V_{1}, V_{2}, \ldots, V_{r} \} \in NC(n)$, then
\begin{eqnarray} \label{eqn:multi4.1}
\varphi_{\pi}[a_{1},a_{2},\ldots, a_{n}]:= \varphi(V_1)[a_{1},a_{2},\ldots,a_{n}]\cdots \varphi(V_r)[a_{1},a_{2},\ldots,a_{n}],
\end{eqnarray}
where 
$\varphi(V)[a_{1},a_{2},\ldots, a_{n}]:= \varphi_{s}(a_{i_1},a_{i_2},\ldots,a_{i_s})=\varphi(a_{i_1}a_{i_2}\cdots a_{i_s})$ for 
$V=\{i_1, i_2, \ldots , i_s\}$ with $i_1 < i_2 < \cdots < i_s$. %Let $(\mathcal{A},\varphi)$ be a non-commutative probability space. 
Note that the order of the variables have been preserved. Also note that the two types of braces $(\ \  )$ and $[\ \  ]$ in (\ref{eqn:mmt4.1}) and (\ref{eqn:multi4.1}) have different uses. In particular, if $\one_n$ denotes the $1$-block partition of $\{1, \ldots , n\}$ then 
\begin{equation} \varphi_{\one_{n}}[a_1,a_2,\ldots, a_n] = \varphi_{n}(a_1,a_2,\ldots, a_n) = \varphi(a_1a_2\cdots a_n).
\end{equation}
The \textit{joint free cumulant} of order $n$ of $(a_1, a_2,\ldots, a_n)$ is 
\begin{equation} \label{eqn:freecumdefn2}
\kappa_{n}(a_{1},a_{2},\ldots, a_{n}) = \sum_{\sigma \in NC(n)} \varphi_{\sigma}[a_{1},a_{2},...,a_{n}]\mu(\sigma, \one_{n}),
\end{equation}
where $\mu$ is the M\"{o}bius function of $NC(n)$. 
It is called a \textit{mixed free cumulant} if at least \textit{one} pair $a_i, a_j$ are different and $a_i \neq  a_j^{*}$ for some $i\neq j$. 
For any $\epsilon_i \in\{1,*\}, 1 \leq i \leq n$,  $\kappa_n (a^{\epsilon_1}, a^{\epsilon_2},\ldots, a^{\epsilon_n})$ is called a \textit{marginal free cumulant} of order $n$ of $\{a,a^{*}\}$. For a self-adjoint element $a$, 
$$\kappa_{n}(a):=\kappa_n(a,a,\ldots, a)$$ is called the $n$-th free cumulant of $a$. 
The free cumulants $\kappa_n$ in (\ref{eqn:freecumdefn2})	are also multi-linear. In particular, for any variables $\{a_i, b_i\}$ and constants $\{c_i\}$, 
\begin{eqnarray*}
\kappa_n(a_1+b_1, \ldots, a_n+b_n)&=&\kappa_n(a_1, \ldots, a_n)+\kappa_n(a_1, b_2, a_3, \ldots, a_n)%\\
%& &\text{\hspace{.51in}} 
+\cdots +\kappa_n(b_1,  \ldots, b_n),\ \ \text{and}\\
\kappa_n(c_1a, c_2a, \ldots, c_na)&=&c_1c_2\cdots c_n\kappa_n(a, a, \ldots, a)=c_1c_2\cdots c_n\kappa_n(a). \end{eqnarray*}
  
Let $\{\kappa_{\pi}: \pi \in NC(n), n\geq 1\}$ be the multiplicative extension of $\{\kappa_n, n\geq 1\}$. 
By using (\ref{eqn:freecumdefn2}), for any $\pi \in NC(n)$, \ $n \geq 1$, 
\begin{equation}\label{eqn: freecumdefn1}
\kappa_{\pi}[a_{1},a_{2},...,a_{n}]:= \sum_{\sigma \in NC(n):\ \sigma \leq \pi} \varphi_{\sigma}[a_{1},a_{2},\ldots, a_{n}]\mu(\sigma, \pi).
\end{equation}
Note that
%\begin{eqnarray}
\begin{eqnarray*}\kappa_{\one_n}[a_1,a_2,\ldots, a_n] &=& \kappa_{n}(a_1,a_2,\ldots, a_n), \\
% \ \text{and} \ \ %\nonumber \\
\kappa_2(a_1, a_2)&=&\varphi(a_1a_2)-\varphi(a_1)\varphi(a_2).
\end{eqnarray*} %\nonumber 
%where $\mu$ is the  M\"{o}bius function of $NC(n)$.  
%\end{eqnarray}
%The number $\kappa_2(a_1, a_2)$ is called the \textit{covariance} between $a_1$ and $a_2$. Note that, in general, 
%$$\kappa_2(a_1, a_2) \neq \kappa_2(a_2, a_1).$$ Equality holds if and only if $\varphi$ is tracial. 
\noindent Using the M\"{o}bius function $\mu$, it can be shown that  
\begin{eqnarray}
\varphi(a_1a_2\cdots a_n) &=& \sum_{\sigma\in NC(n): \  \sigma \leq \one_n} \kappa_{\sigma}[a_1,a_2,\ldots, a_n], \label{eqn:momcum}\\
%\label{eqn: m1to14.1}
%\begin{eqnarray}\label{eqn:momfreecum}
\varphi_{\pi}[a_1, a_2, \ldots, a_n] &=& \sum_{\sigma\in NC(n): \ \sigma \leq \pi} \kappa_{\sigma}[a_1,a_2,\ldots, a_n].\nonumber %\label{eqn: m1to14.22}
\end{eqnarray}

\noindent In particular, (\ref{eqn:freecumdefn2}) and  (\ref{eqn:momcum}) establish a one-to-one correspondence between free cumulants and moments. These relations will be referred  to as \textit{moment-free cumulant relations}. 

Variables $\{a_i: i \in I\}$ are said to be free if their mixed free cumulants are all zero. Similarly, variables $\{a_i^{(p)}:\ i \in I\}$ are said to be asymptotically free if, as $p \to \infty$, they converge jointly to $\{a_i:\ i \in I\}$  which are free.

%\section{Free probability preliminaries}
%\textbf{Incomplete}
%Since we shall work with $p\times p$ matrices, we 
\section{Main results} 
%We shall work with the $*$-probability space $(\mathcal{M}_p(\mathbb{C}), \varphi_p)$ of 
%$p\times p$ random matrices: 
 %\begin{equation}\label{eq:alg}
%\mathcal{M}_p(\mathbb{C})=\{A: A=((a_{ij}))_{p\times p} \ \text{and}\ \ \ {\rm{E}}|a_{ij}|^k < \infty\ \ \text{for all}\ \ i, j, k\},
%\end{equation}
%where the tracial and positive state $\varphi$ is defined as 
%$$\varphi_p(A)=p^{-1}{\rm{E}}\text{Trace}(A).$$
%$\boldsymbol{[}k\boldsymbol{]} = \{1,2,\ldots, k\}$.
%??Define just before using for the first time. Not here.

%???change these notation. No need to list either. use them as they arise. 
%\begin{eqnarray}
%m_y &=& \text{Mar\v{c}enko-Pastur random variable with parameter $y$}, \nonumber \\
%g_{\lambda, \mu} &=& \text{Compound free Poisson random variable with rate $\lambda$ and jump distribution $\mu$}, \nonumber  \\
%\delta_{\alpha} &=& \text{degenerated random variable at $\alpha$}, \nonumber \\
%b_{p} &=& \text{Bernoulli random variable with success probability $p$}, \nonumber\\
%\mu_{a} &=& \text{Probability distribution of the random variable $a$}. \nonumber 
%\end{eqnarray}
%Note that $g_{\lambda, \mu_{\delta_{\alpha}}}$ is a free Poisson random variable with rate $\lambda$ and jump size $\alpha$.  Also for $\alpha \in \mathbb{R}$,  $\alpha m_y$ and $g_{y^{-1},\mu_{\delta_{\alpha y}}}$ have identical probability distributions. 

%\section{Convergence of sample cross-covariance matrices}
Let $X_{l}$ and  $Y_{l}$ be $p \times n_l$ random matrices for all $1 \leq l \leq t$.  The $(i,j)$-th entries of $X_{l}$ and $Y_{l}$ are denoted  by $X_{ij, (n_l, p)}^{(l)}$  and $Y_{ij, (n_l, p)}^{(l)}$  respectively. Note that all entries of these matrices may change with $p$ and $n_l$. We shall often suppress this dependence by dropping the subscripts $p$ and $n_l$. We make the following assumption on these entries.
\vskip 5pt

\noindent \textbf{Assumption I}
\vskip 2pt
\noindent (a) For every $p\geq 1$ and $n_l\geq 1$, the pairs of random variables $(X_{ij, (n_l, p)}^{(l)}, Y_{ij, (n_l, p)}^{(l)})$ are independent across $1 \leq i \leq p$, $1 \leq j \leq n_l$, $1 \leq l \leq t$.
\vskip 2pt
\noindent (b) For all $1 \leq i \leq p$, $1 \leq j \leq n_l$, $1 \leq l \leq t$, $p, n_l \geq 1$, 
%%\begin{eqnarray}
%%&& {\rm{E}}(X_{ij, (n, p)}^{(l)}) =  {\rm{E}}(Y_{ij, (n, p)}^{(l)}) = 0, \nonumber \\
%%&& {\rm{E}}[(X_{ij,( n, p)}^{(l)})^2] = {\rm{E}}[(Y_{ij, (n, p)}^{(l)})^2] = 1,\nonumber \\
%%&& {\rm{E}}(X_{ij, (n, p)}^{(l)}Y_{ij, (n, p)}^{(l)}) = \rho. \nonumber 
%%\end{eqnarray}
\begin{eqnarray}
&& {\mathbb{E}}(X_{ij}^{(l)}) =  {\mathbb{E}}(Y_{ij}^{(l)}) = 0, \ \ %\nonumber \\
%&& 
{\mathbb{E}}[(X_{ij}^{(l)})^2] = {\mathbb{E}}[(Y_{ij}^{(l)})^2] = 1, \ \ %\nonumber \\
%&& 
{\mathbb{E}}(X_{ij}^{(l)}Y_{ij}^{(l)}) = \rho_l. \nonumber 
\end{eqnarray}
\vskip 2pt
\noindent (c) $\displaystyle{\sup_{p, n_l \geq 1} \sup_{1 \leq i \leq p }\sup_{1 \leq j \leq n_l}} {\mathbb{E}}\big(|X_{ij}^{(l)}|^k + |X_{ij}^{(l)}|^k \big) < B_k < \infty$ for all $1 \leq l \leq t$ and $k \geq 1$. 
\vskip 2pt
\noindent (d) $n_l = n_l(p) \to \infty$ as $p \to \infty$ such that $n_l^{-1}p \to y_l < \infty$.
\vskip 10pt
\noindent Define the sample cross-covariance matrices
%\begin{eqnarray}
$C_{l} = \frac{1}{n_l} X_{l} Y_{l}^{*}$ for all $1 \leq l \leq t$. %\nonumber 
%\end{eqnarray}
\subsection{Convergence of $\{C_l: 1 \leq l \leq t\}$ when $y_l \neq 0$}
The symbol $\delta$ will be used in two senses: $\delta_{{x}}$ will denote the probability measure which puts all mass at $x$;  on the other hand, $\delta_{xy}$ is defined as
\begin{eqnarray*}
\delta_{xy}= \begin{cases} 
1  \ \ {\rm if}\ \ x=y\\
0  \ \ {\rm if}\ \ x\neq y.
\end{cases}
\end{eqnarray*}
A compound free Poisson variable with rate $\lambda$ and jump distribution $\mu$ will be denoted by ${\rm{P}}(\lambda, \mu)$ or by 
${\rm{P}}(\lambda, X)$ where $X$ is a random variable with probability law $\mu$.

The following variable shall appear in the limit: 

\begin{definition}An element $c$ of a $*$-probability space will be called a \textit{cross-covariance variable} with parameters $\rho$ and $y, 0 < y <\infty$,  if its free cumulants are given by 
%Here is a formula for these free cumulants: let $\boldsymbol{\eta}_k = (\eta_1, \eta_2,\ldots, \eta_k)^\prime \in \{1,*\}^k$ and $k \geq 1$. Then the $k$-th order free cumulant of $\{c_l\}$ is given by:
\begin{eqnarray*}
\kappa_k(c^{\eta_1},c^{\eta_2},\ldots, c^{\eta_k})= \begin{cases} 
y^{k-1}\rho^{S(\boldsymbol{\eta}_k)}  \ \ {\rm if}\ \ \rho\neq 0\\
y^{k-1}\delta_{S(\boldsymbol{\eta}_k) 0}  \ \ {\rm if}\ \ \rho=0.
\end{cases}
\end{eqnarray*}
where 
$$S(\boldsymbol{\eta}_k):= S(\eta_1, \ldots , \eta_k):={\displaystyle{\sum_{\substack{1 \leq u \leq k \\ \eta_{k+1}=\eta_1}} 
\delta_{\eta_{u}\eta_{u+1}}}}.$$
%\begin{eqnarray}
%\kappa_k(c^{\eta_1},c^{\eta_2},\ldots, c^{\eta_k}) = y^{k-1}(\rho^{\Upsilon(\boldsymbol{\eta}_k)}(1-\delta_{\rho 0}) + \delta_{\rho 0}\delta_{\Upsilon(\boldsymbol{\eta}_k) 0})\ \ \text{where}\ \Upsilon(\boldsymbol{\eta}_k) = {\displaystyle{\sum_{\substack{1 \leq u \leq k \\ j_{k+1}=j_1}} \delta_{\eta_{j_u}\eta_{j_{u+1}}}}}. \nonumber
%\end{eqnarray} 
\end{definition}
%It may be noted that all odd-order free cumulants of a cross-covariance variable vanish. 
It is interesting to note what happens for the two special case $\rho=1$ and $\rho=0$. 

(i) \ When $\rho=1$, $c$ is self-adjoint and its free cumulants are given by:
$$\kappa_k(c)=y^{k-1}.$$
Thus $c$ is a compound free Poisson variable ${\rm{P}}(y^{-1}, \delta_{\{y\}})$ with rate $1/y$ and jump distribution $\delta_{\{y\}}$. The moments of $c$ determine the Mar\v{c}enko-Pastur probability law with parameter $y$. 

(ii) \ When $\rho=0$,
%. Then formula ?? for the free cumulant the free Theorem \ref{thm: 1}(b) implies that 
% for each $1 \leq l \leq t$, 
all odd order free cumulants of $c$ vanish. Moreover, only alternating free cumulants of even order survive, and are given by 
\begin{equation}
\kappa_{2k}(c, c^{*},c, c^{*}\ldots, c^{*})  = \kappa_{2k}(c^{*}, c, c^{*}, c,\ldots, c) = y^{2k-1},\ \forall\ k \geq 1. \label{eqn: traceR}
\end{equation}
%\begin{eqnarray}
%&&\kappa_{2k}(c, c^{*},c, c{*}\ldots,^c^{*}) \nonumber \\
%&& \hspace{2 cm} = \kappa_{2k}(c^{*}, c, c^{*}, c,\ldots, c) = y^{2k-1},\ \forall\ k \geq 1. \label{eqn: traceR}
%\end{eqnarray}
Hence $c$ is a \textit{tracial $R$-diagonal} element. See \cite{NS2006} for the definition and  properties of $R$-diagonal elements. 
It is related to a Mar\v{c}enko-Pastur variable $M_y$ in the following way. 
Let $\tilde{M}_y$ be a \textit{symmetrized Mar\v{c}enko-Pastur} variable with parameter $y$. That is, 
\begin{eqnarray}
\kappa_{k}(\tilde{M}_y) =  \begin{cases} \kappa_{k}(\tilde{M}_y) =y^{k-1}\ \ \text{if $k$ is even},  \\
0\ \ \text{if $k$ is odd}.
\end{cases}\nonumber 
\end{eqnarray}
%all its odd free cumulants are $0$,  and the even free cumulants are same as those of a 
% Mar\v{c}enko-Pastur variable with paramater $y$. Thus  
%$$\kappa_{2k}(X) = y^{2k-1},\  \text{for \a all}\  k \geq 1.$$  %Let $M_y$ and $\tilde{M}_y$ be respectively Mar\v{c}enko-Pastur and symmetrized Mar\v{c}enko-Pastur variables. Also let $Z$ be a random variable with mass $0.5$ at $1$ and $-1$. Moreover, assume $M_y$ and $Z$ are independent and commutative. Then $ZM_y$ and $\tilde{M}_y$ are identically distributed. 
Suppose $u$ is Haar unitary and free of $M_y$ and  $\tilde{M}_y$, then it is easy to see that the $\ast$-distributions of the three variables $c$ (where $\rho=0$),  $u\tilde{M}_y$, and $uM_y$ are identical.
%Let $m_y$ denotes a Mar\v{c}enko-Pastur random variable with parameter $y$ and $Z$ be a random variable with mass $0.5$ at $1$ and $-1$. Also assume $m_y$ and $Z$ are independent and commutative. We refer the probability distribution of $m_yZ$ as symmetrized Mar\v{c}enko-Pastur law. 

%It is easy to see that 

%??Question. can $c$ be expressed in terms of some ``known'' variables, as in the circular variable case??? (NOT sure. Cannot do it)

%$m_y$ denotes a Mar\v{c}enko-Pastur random variable with parameter $y$ and 
%$\mu_{a}$ be the Probability distribution of the random variable $a$.  

The following theorem states the joint convergence and asymptotic freeness of independent cross-covariance matrices when $y_l \neq 0,\ \forall\ 1 \leq l \leq t$. 
\begin{theorem} \label{thm: 1}
Suppose $(X_l, Y_l)$, $1\leq l \leq t$ are pairs of $p\times n_l$ random matrices with correlation parameters $\{\rho_l\}$ and whose entries satisfy Assumption I. Suppose $n_l, p\to \infty$ and  $p/n_{l}\to y_l, 0 < y_l < \infty$ for all $l$. 
%with $y>0$. 
Then the following statements hold for the $p\times p$ cross-covariance matrices $\{C_{l}:=n^{-1}X_lY_l^{*}\}$.
\vskip 3pt
\noindent (a) As elements of the $C^{*}$ probability space
$(\mathcal{M}_p(\mathbb{C}), \varphi_p)$,  $\{C_{l}\}$ converge jointly to free variables 
$\{c_l\}$  where each $c_l$ is a cross-covariance variable with parameters $(\rho_l, y_l)$. The convergence also holds with respect to the state $\tilde\varphi_p$ almost surely. The limiting state is tracial. \vskip 3pt

\noindent (b) Let $\Pi:=\Pi(\{C_{l}:\ 
1 \leq l \leq t\})$ be  any  finite degree real matrix polynomial  in $\{C_{l}, C_{l}^{*}:\ 1 \leq l \leq t\}$ and which is symmetric. Then the ESD of $\Pi$
%(\{C_{l}:\ 1 \leq l \leq t\})$ 
%exists almost surely ??have you proved almost surely??) define ESD LSD??? and is the same as 
converges weakly almost surely to the compactly supported probability law of the self-adjoint variable 
$\Pi(\{c_{l}:\ 1 \leq l \leq t\})$. 
%$\Pi(\{c_{l,y,\rho}:\ 1 \leq l \leq t\})$.
\end{theorem}
Before we prove the theorem, let us make some remarks and give a few examples.   
\begin{remark} 
%If $c_{l,y,\rho}$ ($c_l$ in brief) denote the limits of $C_l$, then its free-cumulants are as follows: 
1. Since $\{c_l\}$ are free in Theorem \ref{thm: 1}, their joint free cumulants can be written in principle using the marginal free cumulants. 

\noindent 2. \cite{CC2004} proved the asymptotic freeness of the sample covariance matrices for the case $y_l \neq 0,\ \forall\ 1 \leq l \leq t$. This result follows 
from Theorem \ref{thm: 1}(a) if we take $\rho = 1$,  since in that case, each $X_l=Y_l$ almost surely and each $C_l$ is a sample covariance matrix. 
%Suppose Assumption I holds and $y>0$.  Then the 
%and Theorem 3.1 of \cite{BB2021}.
%Suppose Assumption I holds and $y>0$.  
% and their limits are self-adjoint. 
%$\{C_l:\ 1 \leq l \leq t\}$  independent sample covariance matrices. 
%Hence 
%Theorem \ref{thm: 1}(a) implies that these independent sample covariance matrices are asymptotically freeness of independent sample covariance matrices. 
%are asymptotically free. For each $1 \leq l \leq t$,  the matrix $C_l$  is also symmetric and hence its limit $c_{l,y,1}$ is self adjoint. Theorem \ref{thm: 1}(b) implies that 
%The free cumulants of $c_l$ \begin{eqnarray}
%\kappa_k(c_{l,y,1}) = y^{k-1},\ \ \forall\ k \geq 1. \nonumber 
%\end{eqnarray}
%Thus, by  Theorem \ref{thm: 1}(c),  the LSD of $C_l$ is the Mar\v{c}enko-Pastur law with parameter $y$. 

%\noindent 3. The convergence can be upgraded to almost sure convergence, with respect to the state $\tilde{\varphi}_p$. This is done by an estimation of the higher moments and applying Borel-Cantelli Lemma. We omit the details. See ???for similar computations.   
\end{remark}

We shall use the following notation: 
\begin{eqnarray*}
\#A &=& \text{Number of elements of the set}\ A,\\
 {\rm{NC}}(k) &=& \text{Set of non-crossing partitions of} \ \{1, \ldots , k\},\\
K(\pi)&=& \text{Kreweras complement of a non-crossing partition}\ pi, \\
|\pi| &=& \text{Number  of  blocks of  the  partition}\ \pi. 
\end{eqnarray*} For two random variable $X$ and $Y$, $X \stackrel{\mathcal{D}}{=} Y$ will mean that they have identical probability laws. In the following discussion, if we are dealing with only one sequence of matrices, we shall drop the index $l$. Similarly, if we are working with any two indices, then without loss of generality, we shall take them to be $1$ and $2$. 

\begin{Example}
%\begin{cor} \label{cor: 1} \textbf{Special cases for $\rho=1$ and $0$}. This should be remark, not corollary. 
%\begin{enumerate}
%\item 
Suppose $\rho_l=0$, $l=1, 2$. Since  $c_1$ and $c_2$ are then free and tracial $R$-diagonal, 
by Theorem 15.17 of \cite{NS2006}, 
$c_1c_2^{*}$ is also tracial $R$-diagonal. 
The free cumulants of $c_1c_2^{*}$  can be calculated using this fact as follows. First note that all free cumulants except 
the even order alternating free cumulants are $0$. These alternating free cumulants are given by
%??drop this???what is the point of this part??? or maybe give as example but then we should do the case where $y's$ are different... By Theorem \ref{thm: 1}(a) and Corollary  \ref{cor: 1}(2),  for each $1 \leq l \leq t$, $c_{l,y,0}$ is tracial R-diagonal element with (\ref{eqn: traceR})
%and  $\{c_{l,y,0}:\ 1 \leq l \leq t\}$ are free across $1 \leq l \leq t$. Then by Theorem 15.17 of \cite{NS2006}, $c_{1,y,0}c_{2,y,0}^{*}$ is tracial R-diagonal element with
\begin{eqnarray}
&& \kappa_{2k} (c_1c_2^{*}, c_2c_1^{*},\ldots, c_1c_2^{*}, c_2c_1^{*}) = \kappa_{2k} ( c_2c_1^{*},\ldots, c_1c_2^{*}, c_2c_1^{*}, c_1c_2^{*})  \nonumber \\
&=& \sum_{\substack{\pi,\sigma \in {\rm{NC}}(k)\\ \sigma \leq K(\pi)}} \left(\prod_{V\in \pi} y_1^{2(\# V)-1} \right) \left( \prod_{W \in \sigma} y_2^{2 (\# W)-1}\right),\nonumber \\
&& \hspace{3 cm} \ \text{(by (\ref{eqn: traceR}) and Exercise 15.24(2) of \cite{NS2006}}) \nonumber \\
&=&  %\sum_{\substack{\pi,\sigma \in {\text{NC}}(k)\\ \sigma \leq K(\pi)}}  y^{4k-\# \pi - \# \sigma} = 
y_1^{k}y_2^{k}\sum_{\pi \in {\rm{NC}}(k)} y_1^{k-|\pi|} \sum_{\sigma \leq K(\pi)} y_2^{k - |\sigma|} \nonumber \\
&=& \sum_{\pi \in {\rm{NC}}(k)} \kappa_{\pi}[y_1M_{y_1}, \ldots, y_1M_{y_1}] \varphi_{K(\pi)}[y_2M_{y_2},\ldots, y_2M_{y_2}]
%\nonumber \\
%&& \hspace{4 cm}
 \ \ \text{($M_{y_1}$ and $M_{y_2}$ are free)} % and both of them have identical probability distribution as $yM_y$} % $\mu_{ym_{y}}$} 
\nonumber \\
&=& \varphi(({y_1M_{y_1}y_2M_{y_2}})^{k}) = \varphi(((\sqrt{y_1M_{y_1}}y_2M_{y_2}\sqrt{y_1M_{y_1}})^{1/2})^{2k}) %=  y^{2k}\lim \frac{1}{p}{\rm{E Tr}}((n^{-2}Y_1Y_1^{*}Y_2Y_2^{*})^k) \nonumber \\
%&=&  y^{2k}\lim \frac{n}{p}\left(\frac{p}{n}\right)^{2k} \frac{1}{n} {\rm{E Tr}}((p^{-2}Y_1^{*}Y_2Y_2^{*}Y_1)^k)  %\nonumber \\
%&=& 
%= y^{4k-1}\varphi(((c_{1})(c_{1})^{*})^k). 
\label{eqn: traceR1}
\end{eqnarray}
%In general, for all $1 \leq l_1 \neq l_2 \leq t$, $c_{l_1,y,0}c_{l_2,y,0}^{*}$ is tracial R-diagonal element with (\ref{eqn: traceR1}).
%\end{enumerate}
%Let ${\rm{P}}(\lambda, \mu)$ denote a compound free Poisson variable with rate $\lambda$ and jump distribution $\mu$. We shall also denote a compound free Poisson variable by ${\rm{P}}(\lambda, X)$ where $X$ is a random variable with probability law $\mu$. 
 %is identical with the distribution of the random variable $X$.  
Thus for $\rho=0$,  the $\ast$-distributions of $c_1c_2^{*}$ and $u {\rm{P}}(1,(\sqrt{y_1M_{y_1}}y_2M_{y_2}\sqrt{y_1M_{y_1}})^{1/2})$ are identical where  $u$ is  Haar unitary and is free of ${\rm{P}}(1,(\sqrt{y_1M_{y_1}}y_2M_{y_2}\sqrt{y_1M_{y_1}})^{1/2})$.  Also note that $y M_y$ is itself a compound free Poisson variable % \stackrel{\mathcal{D}}{=} 
${\rm{P}}(y^{-1},\delta_{\{y^2\}})$.
% are identically distributed.
%and ${\rm{P}}(1,(\sqrt{y_1M_{y_1}}y_2M_{y_2}\sqrt{y_1M_{y_1}})^{1/2})$ are free.
\end{Example}
%\end{cor}

%\begin{cor} \label{cor: 2} \textbf{LSD of some symmetric polynomials}.

%\begin{remark} \textbf{LSD of some symmetric polynomials}.

\begin{Example} By Theorem \ref{thm: 1}(b), the ESD of $C+C^{*}$ converges weakly almost surely to the law whose free cummulants are given by
%By Theorem \ref{thm: 1}(c) , its free cumulants 
%are given by ???please change notation. I dont like upsilon. Also I think dividing by 2 will make the formula nicer and at the same time agree with MP law. ..
\begin{eqnarray}
\kappa_{k}(c+c^{*}) =  y^{k-1}\sum_{\boldsymbol{\eta}_k \in \{1,*\}^k}(\rho^{S(\boldsymbol{\eta}_k)}(1-\delta_{\rho 0}) + \delta_{\rho 0}\delta_{S(\boldsymbol{\eta}_k) 0}),\ \forall\ k \geq 1. \nonumber 
\end{eqnarray}
%%\begin{eqnarray}
%%\kappa_{k}(c_{1,y,\rho}+c_{1,y,\rho}^{*}) =  y^{k-1}\sum_{\boldsymbol{\eta}_k \in \{1,*\}^k}(\rho^{\Upsilon(\boldsymbol{\eta}_k)}(1-\delta_{\rho 0}) + \delta_{\rho 0}\delta_{\Upsilon(\boldsymbol{\eta}_k) 0}),\ \forall\ k \geq 1. \nonumber 
%%\end{eqnarray}
%%Suppose Assumption I holds and $y>0$. Then by Theorem \ref{thm: 1}(c), for $1 \leq l \leq t$, the LSD of 
%%$X_lY_l^{*} + Y_lX_l^{*}$ exists almost surely and they are identical across $1 \leq l \leq t$.  Moreover, by Theorem \ref{thm: 1}(c), this LSD is identical with the probability distribution of $(c_{1,y,\rho}+c_{1,y,\rho}^{*})$ and its free cumulants are
%Let this LSD is identical with the probability distribution of $z_{1,y,\rho}$. Then Theorem \ref{thm: 1}(c) implies that the free cumulants
% $Z_1 = X_1Y_1^{*} + Y_1X_1^{*}$. l imit  $z_{1,\rho}$
%%For the special case $\rho =1$, $C=C^*$ and the free cumulants of $c$ are 
%%\begin{eqnarray}
%%\kappa_{k}(c) =  y^{k-1},\ \forall\ k \geq 1.\hspace{0.5cm} \nonumber 
%%\end{eqnarray}
%%This is nothing but the free cumulants of the Mar{\v{c}}enko-Pastur law.
%% of $c$ distributions of $c_{1,y,1}+c_{1,y,1}^{*}$ and $g_{y^{-1},\mu_{\delta_{2y}}}$ (i.e. $2m_y$) are identical.
%free Poisson with rate $y^{-1}$ and jump size $2y$. Equivalently,   $c_{1,y,1}+c_{1,y,1}^{*}$ and $2g_y$ have identical probability distribution where $g_y$ is the Mar\v{c}enko-Pastur distribution with parameter $y$.
For $\rho =0$, the  above formula can be simplified to: 
%LSD of $C+C^{*}$ has the following free cumulants:
\begin{eqnarray}
\kappa_{k}(c+c^{*}) =  \begin{cases} 2y^{k-1}\ \ \text{if $k$ is even}  \\
0,\ \ \text{if $k$ is odd}
\end{cases},\ \forall\ k \geq 1.\hspace{0.5cm} \nonumber 
\end{eqnarray}
%%\begin{eqnarray}
%%\kappa_{k}(c_{1,y,0}+c_{1,y,0}^{*}) =  \begin{cases} 2y^{k-1}\ \ \text{if $k$ is even}  \\
%%0,\ \ \text{if $k$ is odd}
%%\end{cases},\ \forall\ k \geq 1.\hspace{0.5cm} \nonumber 
%%\end{eqnarray}
%Therefore, $(c_{1,y,0}+c_{1,y,0}^{*})$ and 

%Let $Z$ be a random variable with mass $0.5$ at $1$ and $-1$. Also assume $m_y$ and $Z$ are independent and commutative. We refer $\mu_{m_yZ}$ as symmetrized Mar\v{c}enko-Pastur law.  It is easy to see that 
\noindent Thus the LSD of $C+C^{*}$ is the free additive convolution $\mu\boxplus \mu$ where $\mu$ is the symmetrized Mar\v{c}enko-Pastur law with parameter $y$.

%This is the distribution of ???write in a better way. 
%??$g_{2y^{-1},\mu_{\delta_{y}}}(2b_{0.5}-1)$  
%$XY$ where $X$ is a Bernoulli random variables taking values $0, 1$ with probability $1/2$ each and $Y$ ...and $X$ and $Y$ are independent. 

%have identical probability distributions, where $b_{0.5}$ and $g_{2y^{-1},\mu_{\delta_y}}$ are independently distributed and commutative. 
\end{Example}

%%\item ??drop this one???
%%Suppose Assumption I holds and $y>0$. Then by Theorem \ref{thm: 1}(c), the LSD of $\sum_{l=1}^{t}(X_lY_l^{*} + Y_lX_l^{*})$ exists almost surely and is identical with the probability distribution of $\sum_{l=1}^{t}(c_{l,y,\rho} + c_{l,y,\rho}^{*})$.  Moreover, as $\{c_{l,y,\rho} + c_{l,y,\rho}^{*}:\ 1 \leq l \leq t\}$ are free across $1 \leq l \leq t$, we have
%%\begin{eqnarray}
%%\mu_{\sum_{l=1}^{t}(c_{l,y,\rho} + c_{l,y,\rho}^{*})} \stackrel{\mathcal{D}}{=} \displaystyle{\boxplus_{_{l=1}}^{t}}\mu_{_{c_{l,y,\rho} + c_{l,y,\rho}^{*}}}. \nonumber 
%%\end{eqnarray}

%\item Suppose Assumption I holds and $y>0$. Then by Theorem \ref{thm: 1}(c), 

\begin{Example} By Theorem \ref{thm: 1}(b), the ESD of $CC^{*}$ converges  
%$X_lY_l^{*}Y_lX_l^{*}$ exists 
almost surely. The limit law is the law of the self-adjoint variable $cc^*$. 
%and they are identical across $1 \leq l \leq t$. This LSD is also identical with the probability distribution of $c_{1,y,\rho}c_{1,y,\rho}^{*}$. 
For $\rho \neq 0$,  neither the moments nor the free cumulants of $cc^*$ 
seem to have a simple expression. 
%of $c_{1,y,\rho}c_{1,y,\rho}^{*}$ 
%obtained from  cannot be further simplified. We have simplified form 
However, we know that  when $\rho =0$, $c$ is tracial $R$-diagonal. Hence using Proposition 15.6(2) of \cite{NS2006} and (\ref{eqn: traceR}), we have 
%%\begin{eqnarray}
%%\kappa_{k}(c_{1,y,0}c_{1,y,0}^{*}) &=& \sum_{\pi \in \text{NC}(k)} \prod_{V \in \pi}  y^{2|V|-1} = \sum_{r=1}^{k} |\{\pi \in {\text{NC}}(k):\ \pi\ \text{has $r$ blocks}\}|y^{2k-r} \nonumber \\
%%&=&  \sum_{r=1}^{k}\frac{1}{r}{k-1 \choose r-1}{k \choose r-1}y^{2k-r} = \sum_{r=0}^{k-1} \frac{1}{k-r} {k-1 \choose k-r-1}{k \choose k-r-1}y^{k+r} \nonumber \\
%%&=& \sum_{r=0}^{k-1} \frac{1}{k-r} {k-1 \choose r}{k \choose r+1}y^{k+r} = \sum_{r=0}^{k-1} \frac{1}{r+1} {k-1 \choose r}{k \choose r}y^{k+r} \nonumber
%%\end{eqnarray}
\begin{eqnarray*}
\kappa_{k}(cc^{*}) &=& \sum_{\pi \in {\rm{NC}}(k)} \prod_{V \in \pi}  y^{2\#V-1}\\
& = &\sum_{r=1}^{k} \#\{\pi \in {\rm{NC}}(k):\ \pi\ \text{has $r$ blocks}\}y^{2k-r} \nonumber \\
&=&  \sum_{r=1}^{k}\frac{1}{r}{k-1 \choose r-1}{k \choose r-1}y^{2k-r}\\
& = &\sum_{r=0}^{k-1} \frac{1}{k-r} {k-1 \choose k-r-1}{k \choose k-r-1}y^{k+r} \nonumber \\
&=& \sum_{r=0}^{k-1} \frac{1}{k-r} {k-1 \choose r}{k \choose r+1}y^{k+r}= \sum_{r=0}^{k-1} \frac{1}{r+1} {k-1 \choose r}{k \choose r}y^{k+r} \nonumber
\end{eqnarray*}
which is the $k$-th moment of 
%??Use simple notation...
$yM_y$. % i.e. $g_{y^{-1}, \mu_{\delta_{y^2}}}$.  
That is, 
%Hence, for all $1 \leq l \leq t$, 
the LSD of $n^{-2}XY^{*}YX^{*}$ is the law of the compound free Poisson variable ${\rm{P}}(1, yM_y)$.
%compound free Poisson distribution with rate $1$ and jump distribution $\mu_{ym_y}$. % ??use simpler notation.  
\end{Example}

%%\item Suppose Assumption I holds and $y>0$. Then by Theorem \ref{thm: 1}(c), the LSD of $\sum_{l=1}^{t}X_{l}Y_{l}^{*}Y_{l}X_{l}^{*}$ exists almost surely and is identical with the probability distribution of $\sum_{l=1}^{t}c_{l,y,\rho} c_{l,y,\rho}^{*}$.  Moreover, as $\{c_{l,y,\rho} c_{l,y,\rho}^{*}:\ 1 \leq l \leq t\}$ are free across $1 \leq l \leq t$, by Corollary \ref{cor: 2}(3), 
%%\begin{eqnarray}
%%\mu_{\sum_{l=1}^{t}c_{l,y,\rho} c_{l,y,\rho}^{*}} \stackrel{\mathcal{D}}{=} \displaystyle{\boxplus_{_{l=1}}^{t}}\mu_{_{c_{l,y,\rho} c_{l,y,\rho}^{*}}}. \nonumber 
%%\end{eqnarray}

\begin{Example} 
By Theorem \ref{thm: 1}(b), the LSD of $n^{-2}(X_{1}Y_{1}^{*}Y_{2}X_{2}^{*}+X_{2}Y_{2}^{*}Y_{1}X_{1}^{*})$ exists almost surely.
% and they are identical across $1 \leq l_1 \neq l_2 \leq t$. This LSD is also identical with the probability distribution of $(c_{1,y,\rho}c_{2,y,\rho}^{*}+c_{2,y,\rho}c_{1,y,\rho}^{*})$. 
For $\rho \neq 0$,  the moment or free cumulant sequence of $(c_{1}c_{2}^{*}+c_{2}c_{1}^{*})$ obtained from Theorem \ref{thm: 1}(a) cannot be further simplified. However for $\rho =0$,  recall that 
%By Corollary \ref{cor: 1}(3),  
$c_{1}c_{2}^{*}$ is tracial R-diagonal elementand (\ref{eqn: traceR1}) holds. Therefore, 
\begin{eqnarray}
\kappa_{k}(c_{1}c_{2}^{*}+c_{2}c_{1}^{*}) = \begin{cases} 2\varphi(((\sqrt{y_1M_{y_1}}y_2M_{y_2}\sqrt{y_1M_{y_1}})^{1/2})^{k}) 
\ \ \text{if $k$ is even}, \\
0\ \ \text{if $k$ is odd}.
\end{cases} \nonumber 
\end{eqnarray}
Let $\tilde{{\rm{P}}}(\lambda,X)$  denote a symmetrized compound free Poisson variable{\textemdash}its odd free cumulants are $0$ and the even order free cumulants are the same as those of ${\rm{P}}(\lambda, X)$. Then clearly 
% above expression implies that 
the LSD of $n^{-2}(X_{1}Y_{1}^{*}Y_{2}X_{2}^{*}+X_{2}Y_{2}^{*}Y_{1}X_{1}^{*})$ is the free additive convolution $\nu\boxplus \nu$ where 
$\nu$ is the probability law of the self-adjoint variable $\tilde{{\rm{P}}}(1, \sqrt{y_1M_{y_1}}y_2M_{y_2}\sqrt{y_1M_{y_1}})^{1/2})$.

 %$\tilde{{\rm{P}}}(\lambda, \sqrt{y_1M_{y_1}}y_2M_{y_2}\sqrt{y_1M_{y_1}})^{1/2})\boxplus \tilde{{\rm{P}}}(\lambda, \sqrt{y_1M_{y_1}}y_2M_{y_2}\sqrt{y_1M_{y_1}})^{1/2})$ are identical.
%\begin{eqnarray}
%\kappa_{k}(c_{1,y,\rho}c_{2,y,\rho}^{*}+c_{2,y,\rho}c_{1,y,\rho}^{*}) = \begin{cases} 2y^{2k-1}\varphi(((c_{1,y^{-1},0})(c_{1,y^{-1},0})^{*})^{k/2}), \ \ \text{if $k$ is even} \\
%0,\ \ \text{if $k$ is odd}.
%\end{cases} \nonumber 
%\end{eqnarray}
%By Corollary \ref{cor: 2}(4), the probability distribution of $y^4 (c_{1,y^{-1},0})(c_{1,y^{-1},0})^{*}$ is $$\mu_{y^4g_{1,\mu_{(1/y)m_{1/y}}}} = \mu_{g_{1,\mu_{y^{3}m_{1/y}}}} = \mu_{h},\ \ \text{say}.$$
%Suppose the random variable $h$ has probability distribution $\mu_{g_{1,\mu_{y^{3}m_{1/y}}}}$. Also assume $h$ and $b_{0.5}$ are independent and commutative. 
%Therefore, for all $1 \leq l_1 \neq l_2 \leq t$, the LSD of $(X_{l_1}Y_{l_1}^{*}Y_{l_2}X_{l_2}^{*}+X_{l_2}Y_{l_2}^{*}Y_{l_1}X_{l_1}^{*})$ and the probability distribution of $g_{\frac{2}{y},\  \mu_{h}}(2b_{0.5}-1)$ are identical, where $g_{\frac{2}{y},\  \mu_{h}}$ and $b_{0.5}$ are independent and commutative. 
%\end{cor}
\end{Example}
%\newpage

%\begin{theorem}\label{eq:crossjoint}
%Suppose $(X_l, Y_l)$, $1\leq l \leq t$ are independent random matrices whose entries satisfy Assumption Ie. Suppose $p \to \infty$ such that $p/n \to y, 0 < y < \infty$. Then as elements of $(\mathcal{M}_p(\mathbb{C}), \E \tr)$, the variables $\{C_l,\ 1\leq l\leq t\}$ jointly converge in $*$-distribution to  $\{c_1,\cdots,c_t\}$ in some NCP $(\mathcal{A}, \varphi)$ whose joint moments are given by: 
% $$\varphi(c_{\alpha_1}^{\eta_1}\cdots c_{\alpha_k}^{\eta_k})=\sum_{m=0}^{k-1}y^m
%\sum_{\stackrel{\pi=\{(r,s)\}\in NC_2(2k):}{S(\gamma\pi)=m+1,\alpha_r=\alpha_s \ \text{for all}\  (r,s)\in\pi}}\rho^{T(\pi)}
%$$ for all $\eta_1,...,\eta_k\in\{1,*\}$, 
%for all $\alpha_i\leq t, 1\leq i \leq k$, $k\geq1$.
%\hfill %$\Diamondblack$
%\end{theorem}
%ADD FIGURES.
\begin{proof}[Proof of Theorem \ref{thm: 1}] (a) For simplicity, we will prove the result only for the special case where $n_l$, $y_l$ and $\rho_l$ do not depend on $l$. It will be clear from the arguments that the same proof works for the general case. 
Consider a typical monomial 
$$(X_{\alpha_{1}}Y_{\alpha_1}^{*})^{\eta_1}  \cdots (X_{\alpha_k}Y_{\alpha_k}^{*})^{\eta_k}.$$ 
This product has $2k$ factors. We shall write this monomial in a specific way to facilitate computation. 
Note that 
\begin{equation}
(X_{\alpha_s}Y_{\alpha_s}^{*})^{\eta_s}=
\begin{cases}{X_{\alpha_s}Y_{\alpha_s}^{*}}&\ \ \text {if}\ \  \eta_{s}=1,\\
Y_{\alpha_s}X_{\alpha_s}^{*}&\ \ \text{if}\  \ \eta_{s}=*.\\
\end{cases} \nonumber
\end{equation}
For every index $l$, two types of matrices, namely $X$ and $Y$ are involved. To keep track of this, 
define 
\begin{equation}
(\epsilon_{2s-1}, \epsilon_{2s})=
\begin{cases}{(1, 2)}& \text { if }\  \eta_{s}=1,\\
 (2, 1)& \text { if }\  \eta_{s}=*.\\
\end{cases} \nonumber
\end{equation}
Let $\lceil \cdot \rceil$ be the ceiling function. % and $r\%n$ be the remainder after dividing $r$ by $n$. 
Note that
\begin{eqnarray}
\epsilon_{r} \neq \epsilon_{s} \Leftrightarrow \begin{cases} (\epsilon_r, \epsilon_{r+1}) = (\epsilon_{s-1},\epsilon_s) \Leftrightarrow  \eta_{\lceil r/2 \rceil} = \eta_{\lceil s/2 \rceil}\ \ \text{if $r$ is odd, $s$ is even}, \\
(\epsilon_{r-1},\epsilon_r) = (\epsilon_s,\epsilon_{s+1}) \Leftrightarrow \eta_{;\lceil r/2 \rceil} = \eta_{\lceil s/2 \rceil}\ \ \text{if $r$ is even, $s$ is odd}.
\end{cases} %\Leftrightarrow  \eta_{\lceil r/2 \rceil} = \eta_{\lceil s/2 \rceil},\ \ \forall r\%2 \neq s\% 2.\ \  
 \label{eqn: epeta}
\end{eqnarray} %$r ( \mod n)$
%Let $1 \leq r < s \leq 2k$ and $\epsilon_{r} \neq \epsilon_{s}$. Then $(\epsilon_r, \epsilon_{r+1}) = (\epsilon_{s-1},\epsilon_s)$ i.e. $\eta_{(r+1)/2} = \eta_{s/2}$ if $r$ is odd and $s$ is even. Moreover, $(\epsilon_{r-1},\epsilon_r) \neq (\epsilon_s,\epsilon_{s+1})$ i.e. $\eta_{r/2} = \eta_{(s+1)/2}$ if $r$ is even and $s$ is odd. Therefore,
%\begin{eqnarray}
%\epsilon_{r} \neq \epsilon_{s} \Leftrightarrow \eta_{\lceil r/2 \rceil} = \eta_{\lceil s/2 \rceil},\ \ \forall r \neq s. 
%\end{eqnarray}

 %If $r$ is odd, $s$ is even and $\epsilon_{r} \neq \epsilon_{s}$, then  $(\epsilon_r, \epsilon_{r+1}) = (\epsilon_{s-1},\epsilon_s)$ i.e. $\eta_{(r+1)/2} = \eta_{s/2}$. Further, if $r$ is even, $s$ is odd and $\epsilon_r \neq \epsilon_s$, the $(\epsilon_{r-1},\epsilon_r) \neq (\epsilon_s,\epsilon_{s+1})$ i.e. $\eta_{r/2} = \eta_{(s+1)/2}$. 

\noindent Observe that $\delta_{\epsilon_{2s-1}\epsilon_{2s}}=0$ for all $s$.
Define \begin{equation}
A_l^{(\epsilon_{2s-1})}=
\begin{cases}{X_l} & \text { if }\  \epsilon_{2s-1}=\eta_s=1 \\
 Y_l& \text { if }\  \epsilon_{2s-1}=2 \ \text{(or} \ \eta_s=*),\\
\end{cases}
\end{equation}
%\end{equation}
%\begin{equation}
\begin{equation}A_l^{(\epsilon_{2s})}=
\begin{cases}X_{l}^{*} & \text { if }\  \epsilon_{2s}=1 \ \text{(or}\ \eta_s=*) \\
 Y_l^{*}& \text { if }\  \epsilon_{2s}=2 \ \text{(or}\ \eta_s=1).\\
\end{cases} \nonumber
\end{equation}
Extend the vector $(\alpha_1, \ldots ,  \alpha_k)$ of length $k$ to the vector of length $2k$ as $$(\beta_1, \ldots , \beta_{2k}):= (\alpha_1,\alpha_1, \ldots \alpha_k, \alpha_k).$$
We need to show that for all choices of $\alpha_s\in\{1,2,...,t\}$ and $\eta_s\in\{1, *\}$,
$$L_p:=\dfrac{1}{pn^k}
{\mathbb{E}} {\rm{Tr}} (A_{\beta_1}^{(\epsilon_1)}A_{\beta_{2}}^{(\epsilon_2)}\cdots A_{\beta_{2k}}^{(\epsilon_{2k})})$$ converges to the appropriate limit.
Upon expansion, $$L_p=\dfrac{1}{pn^k}\sum_{I_{2k}}\mathbb{E} \prod_{\substack{1 \leq s \leq 2k \\ i_{2k+1}=i_1}} A_{\beta_s}^{(\epsilon_s)}(i_s,i_{s+1})$$
where $A_{\beta}^{(\epsilon)}(i, j)$ denotes the $(i, j)$th element of $A_{\beta}^{(\epsilon)}$ for all choices of $\beta, \epsilon, i$ and $j$, and  
$$I_{2k}=
\{(i_1,i_2, \ldots, i_{2k}): 
\  1\leq i_{2s-1}\leq p,\ 1\leq i_{2s}\leq n,  \  1\leq s\leq 2k\}.$$
%Further define $i_{2k+1}=i_1$.
Observe that the values of these $i_j$ have different ranges $p$ and $n$, depending on whether $j$ is odd or even.

%These two types of indices will be considered different. As before we 
Note that the expectation of any summand is zero if there is at least one $(i_s, i_{s+1})$ whose value is not repeated elsewhere in the product. 
So, as usual, to split up the sum into indices that match, consider any connected bipartite graph between the \textit{distinct} odd and even indices, $I=\{i_{2s-1}:1\leq s\leq k\}$ and $J=\{i_{2s}:1\leq s\leq k\}$. Then we need to consider only those cases where each edge appears at least twice. Hence there can be at most $k$ distinct edges and since the graph is connected, 
$$\#I+\#J\leq \#E+1\leq k+1.$$ By Assumption I, there is a common bound for all expectations involved. Hence, 
the total expectation of the terms involved in this graph is of the order 
$$O(\dfrac{p^{\#I}n^{\#J}}{pn^k})=O(p^{\#I+\#J-(k+1)})$$ since $\dfrac{p}{n}\to y>0$. As a consequence, only those terms can potentially contribute to the limit for which $\#I+ \#J=k+1$. This implies that $\#E=k$. So each edge is repeated exactly twice.
Let 
\begin{equation}
P_{2}(2k)=\{\pi: \pi \ \text{pair\ partition\ of}\ \{1, 2, \ldots , 2k\}\}.
\end{equation}
 Then each edges in $E$ corresponds to  some $\pi=\{(r,s): r < s\} \in P_2(2k)$. Let 
\begin{equation}
a_{r,s}=
\begin{cases}1 &\text { if } \ r, s \text{ are both odd or both even},  \\
 0   &\text { otherwise}.\\
\end{cases}
\end{equation}
Then  we have 
\begin{eqnarray*}\lim_{p\to \infty} L_p &=& \lim_{n\to\infty}\dfrac{1}{pn^k}\sum_{I_{2k}}{\mathbb{E}}\big[\prod_{s=1}^{2k} A_{\beta_{s}}^{(\epsilon_s)}(i_s, i_{s+1})\big] 
= \sum_{\pi\in P_2(2k)}\lim_{n\to\infty}\dfrac{1}{pn^k}\sum_{I_{2k}}\prod_{(r,s)\in\pi} E(r, s)\ \ \text{say},
\end{eqnarray*}
where, suppressing the dependence on other variables,  
\begin{eqnarray*}E(r, s)&=& {\mathbb{E}}\big[A_{\beta_r}^{(\epsilon_r)} (i_r,i_{r+1})A_{\beta_{s}}^{(\epsilon_s)}(i_s, i_{s+1})\big]\\
 &=&\delta_{\beta_r\beta_s}
 (\rho(1-\delta_{\epsilon_r\epsilon_s})+\delta_{\epsilon_r\epsilon_s})(\delta_{i_ri_s}\delta_{i_{r+1}i_{s+1}}a (r,s) +\delta_{i_ri_{s+1}}\delta_{i_si_{r+1}}(1-a (r,s))) \nonumber \\
  &=&\delta_{\beta_r\beta_s}
( \rho^{1-\delta_{\epsilon_r\epsilon_s}}(1-\delta_{\rho 0})+\delta_{\rho 0}\delta_{\epsilon_r\epsilon_s})(\delta_{i_ri_s}\delta_{i_{r+1}i_{s+1}}a (r,s) +\delta_{i_ri_{s+1}}\delta_{i_si_{r+1}}(1-a (r,s))). 
\end{eqnarray*}
Recall that $|\rho|\leq 1$. Hence  each $E(r,s)$ is a sum of two factors{\textemdash}one of them is bounded by $\delta_{i_ri_s}\delta_{i_{r+1}i_{s+1}}$ and the other is bounded by $\delta_{i_ri_{s+1}}\delta_{i_si_{r+1}}$. Hence when we expand $\prod_{(r,s)\in\pi} E(r, s)$, each term  involves a product of these $\delta$-values.
%Now recall our already familiar arguments with (C1) and (C2) constraints. Terms where there is at least one  (C1) constraint (that is $\delta_{i_{r}i_{s}}\delta_{i_{r+1}i_{s+1}}=1$ for at least one $(r,s)$), do not survive in the limit. 
Using arguments similar  to those used in the proof of Theorem 3.2.6 in \cite{bose2018}, it is easy to see 
%See for instance the proofs of 
%Theorems \ref{thm:i.i.d.samplecov} in Chapter \ref{chapter:lsd} and ????. Thus 
that the only term that will survive in 
$\prod_{(r,s)\in\pi} E(r, s)$
is 
$$\prod_{(r,s)\in \pi} \delta_{\beta_r\beta_s} \big( \rho^{1-\delta_{\epsilon_r\epsilon_s}}(1-\delta_{\rho 0})+\delta_{\rho 0}\delta_{\epsilon_r\epsilon_s}\big)\delta_{i_ri_{s+1}}\delta_{i_si_{r+1}}(1-a (r,s)).$$

\noindent Hence $\lim L_p$ is equal to 
 %(\ref{eq:intermediate}) reduces to 
\begin{equation}\label{eq:internc}\sum_{\pi\in {\rm{NC}}_2(2k)}\hspace{-5pt}\lim_{n\to\infty}\dfrac{1}{pn^k}\sum_{I_{2k}}\hspace{-5pt}\prod_{(r,s)\in\pi}\hspace{-5pt}\delta_{\beta_r\beta_s}\big( \rho^{1-\delta_{\epsilon_r\epsilon_s}}(1-\delta_{\rho 0})+\delta_{\rho 0}\delta_{\epsilon_r\epsilon_s}\big)(1-a(r,s))\prod_{r=1}^{2k}\delta_{i_ri_{\gamma\pi(r)}}.
\end{equation}
But for any $\pi\in {\rm{NC}}_2(2k)$ if $(r,s)\in\pi$ then $r$ and $s$ have different parity and hence $a(r,s)=0$. 
Let $\gamma$ denote the cyclic permutation $1\to 2\to \cdots \to 2k\to 1$. Then (\ref{eq:internc}) simplifies to 
\begin{equation}\label{eq:simplifiedxy}\sum_{\pi\in {\rm{NC}}_2(2k)}\prod_{(r,s)\in\pi}\delta_{\beta_r\beta_s}( \rho^{1-\delta_{\epsilon_r\epsilon_s}}(1-\delta_{\rho 0})+\delta_{\rho 0}\delta_{\epsilon_r\epsilon_s})\lim_{n\to\infty}
\dfrac{\#\{I_{2k}:i_r=i_{\gamma\pi(r)}\ \text{for all}\  r\}}{pn^k}.
\end{equation}
Now note that as $\pi\in {\rm{NC}}_2(2k)$, $\gamma\pi$ contains $k+1$ blocks. Moreover, each block of $\gamma\pi$ contains only odd or only even elements.  Let 
\begin{equation*}S(\gamma\pi)=\ \text{Number of blocks in}\ \ \gamma\pi\ \ \text{with only odd elements}. 
\end{equation*} 
Then the number of blocks of $\gamma\pi$ with only even elements is 
$k+1-S(\gamma\pi)$. 
Suppose $\pi\in {\rm{NC}}_2(2k)$ such that $S(\gamma\pi)=m+1$.  Then it is clear that 
\begin{equation*}\#\{I_{2k}:i_r=i_{\gamma\pi(r)}\ \text{for all}\  r\}=p^{m+1}n^{k+1-(m+1)}
\end{equation*} and hence using (\ref{eq:simplifiedxy})
\begin{eqnarray*}
\lim L_p &=&\sum_{\pi\in {\rm{NC}}_2(2k)}\prod_{(r,s)\in\pi}\delta_{\beta_r\beta_s}( \rho^{1-\delta_{\epsilon_r\epsilon_s}}(1-\delta_{\rho 0})+\delta_{\rho 0}\delta_{\epsilon_r\epsilon_s})\lim_{n\to\infty}
\dfrac{\#\{I_{2k}:i_r=i_{\gamma\pi(r)}\forall r\}}{pn^k}\\
&=& \sum_{m=0}^{k-1}y^m
\sum_{\stackrel{\pi\in {\rm{NC}}_2(2k):}{S(\gamma\pi)=m+1}}
\prod_{(r,s)\in \pi}\delta_{\beta_r\beta_s} ( \rho^{1-\delta_{\epsilon_r\epsilon_s}}(1-\delta_{\rho 0})+\delta_{\rho 0}\delta_{\epsilon_r\epsilon_s}) y^{m} \nonumber \\
&=& \sum_{m=0}^{k-1}y^m
\sum_{\stackrel{\pi\in {\rm{NC}}_2(2k):}
{S(\gamma\pi)=m+1}}\prod_{(r,s)\in \pi} \delta_{\beta_r\beta_s}(\rho^{T(\pi)}(1-\delta_{\rho 0}) + \delta_{\rho 0}\delta_{T(\pi) 0}) 
\end{eqnarray*}
where 
\begin{equation*}T(\pi)=\#\{(r,s)\in\pi:\delta_{\epsilon_r\epsilon_s}=0\}\ \ \text{for}\ \ \pi\in {\rm{NC}}_2(2k).
\end{equation*}
Hence we have proved that $\{C_l:\ 1\leq l \leq t\}$ converge jointly in $*$-distribution to say  $\{c_{{l,y,\rho}}:\ 1 \leq l \leq t\}$ which are the limit NCP $(\mathcal{A},\varphi)$. We still have to identify the limit and prove the freeness. For this we need to go from ${\rm{NC}}_2(2k)$ to ${\rm{NC}}(k)$. 
Define 
\begin{eqnarray}
\tilde{J}_i &=& \{j \in \{1,2,\ldots, 2k\}:\ \beta_j=i\},\ \ 1 \leq i \leq t, \nonumber \\
\tilde{B}_k &=& \{\pi \in {\rm{NC}}_2(2k):\ \pi = \cup_{i=1}^{t}\pi_i,\ \pi_i \in {\rm{NC}}_2(\tilde{J}_i ),\ 1 \leq i \leq t\}, \nonumber \\
\tilde{B}_{m,k} &=& \{\pi \in \tilde{B}_k:\ S(\gamma\pi) = m+1\}. \nonumber 
\end{eqnarray}
Note that $\cup_{m=0}^{k-1} \tilde{B}_{m,k} = \tilde{B}_{k}$ and hence
\begin{eqnarray}
\lim L_p = \sum_{m=0}^{k-1}y^{m} \sum_{\pi \in \tilde{B}_{m,k}}(\rho^{T(\pi)}(1-\delta_{\rho 0}) + \delta_{\rho 0}\delta_{T(\pi) 0}) . \label{eqn: 1}
\end{eqnarray}
Also define
\begin{eqnarray}
J_i &=& \{j \in \{1,2,\ldots,k\}:\ \alpha_j = i\},\ \ 1 \leq i \leq t, \nonumber \\
B_k &=& \{\pi \in {\rm{NC}} (k):\ \pi = \cup_{i=1}^{t}\pi_i,\ \pi_i \in {\rm{NC}}({J}_i ),\ 1 \leq i \leq t\}, \nonumber \\
{B}_{m,k} &=& \{\pi \in B_k:\  \pi\ \text{has}\ m\ \text{blocks}\}. \nonumber 
\end{eqnarray}
Note that $\cup_{m=0}^{k-1} {B}_{m+1,k} = {B}_{k}$.  %\vskip 5pt
\noindent For any finite subset $S=\{j_1,j_2,\ldots, j_r\}$ of positive integers, define 
\begin{eqnarray}
\tilde{T}(S) = \sum_{\substack{1\leq u \leq r \\ j_{r+1}=j_1}} \delta_{\eta_{j_u}\eta_{j_{u+1}}} \nonumber 
\end{eqnarray}
and for $\pi = \{V_1,V_2,\ldots,V_m\}  \in {\rm{NC}}(k)$, define
\begin{eqnarray}
\mathcal{T}(\pi) = \sum_{i=1}^{m}\tilde{T}(V_i). \nonumber
\end{eqnarray}
Consider the bijection $f: {\rm{NC}}_2(2k) \to {\rm{NC}}(k)$ as follows. Take $\pi \in {\rm{NC}}_2(2k)$. Suppose $(r,s)$ is a block of  $\pi$. Then $\lceil r/2 \rceil$ and $\lceil s/2 \rceil$ are put in the same block in $f(\pi) \in  {\rm{NC}}(k)$.  
Using arguments similar to used in the proof of Lemma 3.2 in \cite{BB2021}, it is easy to see that $f$ is indeed a bijection and is also a bijection between $\tilde{B}_{m,k}$ and $B_{k-m,k}$.  Moreover, using (\ref{eqn: epeta}), it is immediate that $T(\pi) = \mathcal{T}(f(\pi))\ \forall\ \pi \in \tilde{B}_{m,k}$ i.e. $T(f^{-1}(\pi)) = \mathcal{T}(\pi)\ \forall\ \pi \in {B}_{k-m,k}$. 

As an example,  let $\pi = \{(1,8),(2,5),(6,7),(3,4),(9,10)\}$. Then $\pi  \in \tilde{B}_{2,5}$ and is mapped to 
$f(\pi) = \{(1,3,4),(3),(5)\} \in {B}_{3,5}$. Let $(\eta_1,\eta_2,\eta_3,\eta_4,\eta_5) = (1,*,1,*,1)$. 
Further, $(\epsilon_1,\epsilon_2,\ldots, \epsilon_{10}) = (1,2,2,1,1,2,2,1,1,2)$ and $T(\pi) = \mathcal{T}(f(\pi)) = 3$. %$$. 
 
Now (\ref{eqn: 1}), we have
\begin{eqnarray}
\lim L_p &=& \sum_{m=0}^{k-1}y^{m} \sum_{\pi \in {B}_{k-m,k}}\big(\rho^{T(f^{-1}(\pi))}(1-\delta_{\rho 0}) + \delta_{\rho 0}\delta_{T(f^{-1}(\pi)) 0}\big) \nonumber \\
& = & \sum_{m=0}^{k-1}y^{m} \sum_{\pi \in {B}_{k-m,k}}\big(\rho^{\mathcal{T}(\pi)}(1-\delta_{\rho 0}) + \delta_{\rho 0}\delta_{\mathcal{T}(\pi) 0}\big)   \nonumber \\
&=& \sum_{m=0}^{k-1}y^{k-m-1} \sum_{\pi \in {B}_{m+1,k}}\big(\rho^{\mathcal{T}(\pi)}(1-\delta_{\rho 0}) + \delta_{\rho 0}\delta_{\mathcal{T}(\pi) 0}\big).\ \ \hspace{0.5 cm}  \label{eqn: 2} 
\end{eqnarray}
 Then, (\ref{eqn: 2}) implies
\begin{eqnarray}
\varphi(c_{\alpha_1}^{\eta_1}c_{\alpha_2}^{\eta_2}\cdots c_{\alpha_k}^{\eta_k}) &=& \sum_{m=0}^{k-1}\sum_{\substack{\pi \in {B}_{m+1,k} \\ \pi = \{V_1,V_2,\ldots,V_{m+1}\}}}\prod_{l=1}^{m+1} y^{\#V_l-1} \big(\rho^{\tilde{T}(V_l)}(1-\delta_{\rho 0}) + \delta_{\rho 0} \delta_{\tilde{T}(V_l) 0}\big) \nonumber \\
&=& \sum_{\pi \in B_k} \prod_{l=1}^{\# \pi} y^{\#V_l-1} (\rho^{\tilde{T}(V_l)}(1-\delta_{\rho 0}) + \delta_{\rho 0}\delta_{\tilde{T}(V_l) 0} ). \nonumber
\end{eqnarray}
Hence, by moment-free cumulant relation, we have
\begin{eqnarray}
\kappa_{\pi}[c_{\alpha_1}^{\eta_1},c_{\alpha_2}^{\eta_2},\ldots, c_{\alpha_k}^{\eta_k}] &=& 0\ \text{for all}\ \pi \in {\rm{NC}}(k) - B_k,\ \  \nonumber \\
\kappa_{\pi} [c_{\alpha_1}^{\eta_1},c_{\alpha_2}^{\eta_2},\ldots, c_{\alpha_k}^{\eta_k}] &=& \prod_{l=1}^{\# \pi} y^{\#V_l-1} (\rho^{\tilde{T}(V_l)}(1-\delta_{\rho 0}) + \delta_{\rho 0}\delta_{\tilde{T}(V_l) 0} ) \ \text{for all}\ \pi \in B_k. \nonumber
\end{eqnarray}
This implies that 
\begin{eqnarray}
\kappa_k(c_{\alpha_1}^{\eta_1},c_{\alpha_2}^{\eta_2},\ldots, c_{\alpha_k}^{\eta_k}) = \begin{cases} y^{k-1}
\big(\rho^{\mathcal{T}(\one_k)}(1-\delta_{\rho 0}) + \delta_{\rho 0}\delta_{\mathcal{T}(\one_k) 0}\big),\ \ \text{if $\alpha_1=\alpha_2= \cdots =\alpha_k$}, \\ 
0\ \ \text{otherwise}. \end{cases} \nonumber 
\end{eqnarray}
Therefore $\{c_{l}:\ 1 \leq l \leq t\}$ are free across $l$, and the marginal free cumulant of order $k$ is 
\begin{eqnarray}
\kappa_k(c_{l}^{\eta_1},c_{l}^{\eta_2},\ldots,c_{l}^{\eta_k}) &=& y^{k-1}(\rho^{S(\boldsymbol{\eta}_k)}(1-\delta_{\rho 0}) + \delta_{\rho 0}\delta_{S(\boldsymbol{\eta}_k) 0})\ \nonumber \\
\text{where}\ S(\boldsymbol{\eta}_k) &=& \mathcal{T}(\one_k)= {\displaystyle{\sum_{\substack{1 \leq u \leq k \\ u_{k+1}=u_1}} \delta_{\eta_{u}\eta_{u+1}}}}. \nonumber
\end{eqnarray} 
%???why do you need the Upsilon notation? 
This completes the proof of Theorem \ref{thm: 1} (a)  for the state $\varphi_p$ for the special case when the values of $\rho_l$ and of $y_l$ are same. It is easy to see that the above proof continues to hold for the general case, except for notational complexity. We omit the details.  
 
Now we argue convergence with respect to the state $\tilde{\varphi}_p$. Consider any polynomial $\Pi$. We have already shown convergence of 
$p^{-1}\mathbb{E} [{\rm Trace}(\Pi^k)]$ to say $\beta_k$ for all $k$. %But these are nothing but the moments of the expected empirical spectral distribution (EESD). 
We need to show that $p^{-1}{\rm Trace}(\Pi^k)$ converges to $\beta_k$ almost surely. For this it is enough to show that for every $k$,  
$$\mathbb{E}\Big[p^{-1}{\rm Trace}(\Pi^k)-p^{-1}E [{\rm Trace}(\Pi^k)]\Big]^4=O(p^{-2}).$$
Then an application of the Borel-Cantelli Lemma would finish the proof. Now, a further simplification is that it is enough to prove this for any monomial. Then the above estimate is obtained by a counting argument, which is similar to, but simpler than what has already been used in the proof so far. We omit the details. See Section 2 in the Supplementary material of \cite{BB2016ynot0} for similar arguments.  This complete the proof of (a). 

\noindent (b) Now suppose that $\Pi$ is symmetric. Then by the above argument, all moments of $\Pi$ converge and there is a $C > 0$, depending on $\Pi$ such that the limiting $k$th moment is bounded by $C^{k}$ for all $k$. This implies that these moments define a unique probability law say $\mu$ with support contained in $[-C, \ C]$, and hence the EESD of $\Pi$ converges weakly to $\mu$. Again, the almost sure convergence of the ESD can be established by the arguments discussed above. We omit the details. Note that this argument works only if $\Pi$ is a symmetric matrix. 
\end{proof}
%\textbf{Need to add spacial cases.} Note that if we take  $\rho=1$, then we we get back the result on joint convergence of independet $S$ matrices.

\subsection{Convergence of $\{C_l: 1 \leq l \leq t\}$ when $y =0$}
Now suppose $pn^{-1}\to 0$ as $n,p\to\infty$. Then Theorem \ref{thm: 1} leads to a degenerate distribution when $y=0$. In this case, we need a centering as well as a different scaling for a non-degenerate limit to exist.  Let us quickly recall a known result.  
% In this case, the following result is well known. 
Consider the sample covariance matrix $S=n^{-1}XX^*$ where the  entries $X_{ij}$ of $X$ are independent and ${\mathbb{E}}(X_{11})=0,\ {\rm{Var}}(X_{11})=1, {\mathbb{E}}(X_{11}^4)<\infty$. Then the empirical spectral distribution (ESD) of $\sqrt{np^{-1}}(S-I_p)$ converges weakly almost surely to the standard semi-circle law where $I_p$ is the identity matrix of order $p$. For proof, one can see \cite{bose2018}. This proof actually shows that when all moments are finite, then the moments of the ESD converge to the moments of the semi-circle law and hence there is convergence as elements of $(\mathcal{M}_p(\mathbb{C}), \varphi_p)$ to a semi-circular variable. 

We provide a generalization of this result involving cross-covariance matrices. 
Before we state the result, we need to recall the definition of elliptic variables. 
\begin{definition}\label{def:elliptic}
Suppose $(\mathcal{A}, \varphi)$ is an NCP. An element $e\in \mathcal A$ is said to be an {\it elliptic} variable\index{elliptic variable}\index{variable, elliptic} with parameter $\rho$, $-1\leq \rho\leq 1$ if its free cumulants of order one and of order greater than $2$ are zero and its second order free cumulants are given by 
$$\kappa_2(e, e)=\kappa_2(e^*, e^*)=\rho, \ \ \kappa_2(e, e^*)=\kappa_2(e^*, e)=1.$$
\end{definition}
An elliptic variable has the following representation. Suppose $s_1$ and $s_2$ are two free standard semi-circular variables. Define 
$$e=\sqrt{\dfrac{1+\rho}{2}}s_1+ \sqrt{-1}\sqrt{\dfrac{1-\rho}{2}}s_2.$$
Then $e$ is an elliptic variable with parameter $\rho$. Note that $\rho=1$ and $\rho=0$ yield respectively the standard semi-circular and the standard circular variable. 

 We use the following crucial fact for variables to be elliptic \textit{and} free: 
Variables $\{e_i, 1\leq i \leq t\}$ are elliptic with parameters 
 $\{\rho_i, 1\leq i \leq t\}$ on an NCP $(\mathcal{A}, \varphi)$ and are free if and only if,  for all $k \geq 1$, for all $1\leq i\leq k$, $\epsilon_i \in \{1, *\}$ and  
$\tau_i \in \{1, \ldots , t\}$,  the following holds for the joint moments:
\begin{align}\label{eqn:twomatrixfree2}
\varphi(e_{\tau_1}^{\epsilon_1}e_{\tau_2}^{\epsilon_2}\cdots e_{\tau_{2k}}^{\epsilon_{2k}})=\sum_{\pi \in {\rm{NC}}_2(2k)}\rho_1^{T_1(\pi)}\ldots \rho_m^{T_m(\pi)}\prod_{(r,s)\in \pi}\delta_{\tau_r\tau_{s}}.
\end{align}
where
$$T_{\tau}(\pi):=\#\{(r,s)\in \pi \;:\; \delta_{\epsilon_r\epsilon_s}=1, \tau_r=\tau_s=\tau\}.$$
Note that, it is understood that all odd order moments are $0$. 
\vskip 5pt
\noindent Now we can state our theorem. 

%??Question. Is the result true almost surely? and for different $n_l$ and $y$ and $\rho$? Is it easy to argue?
\begin{theorem}\label{thm: 2} Suppose Assumption I holds with $y_l=0,\ \forall\ 1 \leq l \leq t$.\vskip 3pt
\noindent (a) Then $E_l=\sqrt{n_l p^{-1}}\big(C_l-\rho_l I_p\big), 1\leq l \leq t$ as elements of $(\mathcal{M}_p, \varphi_p)$ converge jointly to free elliptic variables $e_1,...,e_t$, with parameters $\rho_l^2$. The convergence also holds with respect to the state $\tilde\varphi_p$ almost surely. The limiting state is tracial. \vskip 3pt
\noindent (b) Let $\Pi:=\Pi(\{E_{l}:\ 
1 \leq l \leq t\})$ be  any  finite degree real matrix polynomial  in $\{E_{l}, E_{l}^{*}:\ 1 \leq l \leq t\}$ and which is symmetric. Then the ESD of $\Pi$ converges weakly almost surely to the compactly supported probability law of the self-adjoint variable 
$\Pi(\{e_{l}:\ 1 \leq l \leq t\})$. 
\end{theorem}

\begin{Example}
Theorem \ref{thm: 2} is useful to find the LSD of any appropriately centered and scaled symmetric polynomial of $\{C_l:\ 1 \leq l \leq t\}$. For example, consider 
\begin{eqnarray}
\Pi &=& \sqrt{\frac{\min(n_1,n_2)}{p}}(C_1+C_1^{*} + C_1C_2^{*} + C_2C_1^{*} - 2 \rho_1(1+\rho_2)I_p) \nonumber \\
&=& \sqrt{\frac{\min(n_1,n_2)}{p}}\bigg[(C_1+C_1^{*} - 2\rho_1 I_p) + (C_1-\rho_1I_p)(C_2-\rho_2I_P)^{*} + (C_2-\rho_2 I_p)(C_1 - \rho_1 I_p)^{*}  \nonumber \\
&&\hspace{5.75 cm}+ \rho_2(C_1+C_1^{*}-2\rho_1I_p) + \rho_1(C_2+C_2^{*}-2\rho_2I_p)\bigg] \nonumber \\
&=& \sqrt{\min(n_1,n_2)}\bigg[n_1^{-1/2} (E_1+E_1^{*}) + \sqrt{p}n_1^{-1/2}n_2^{-1/2}(E_1E_2^{*} + E_2E_1^{*}) \nonumber \\
&& \hspace{5.25 cm} + n_1^{-1/2}\rho_2 (E_1+E_1^{*}) + n_2^{-1/2}\rho_1(E_2+E_2^{*})\bigg].
% \\ 
%& \to & \bigg(\lim \sqrt{\frac{\min(n_1,n_2)}{n_1}}\bigg)(1+\rho_2)(e_1 +e_1^{*}) +  \bigg(\lim \sqrt{\frac{\min(n_1,n_2)}{n_2}}\bigg)(e_2 +e_2^{*}).
%\label{eqn: sym0}  
\end{eqnarray}
Hence it follows that  
\begin{eqnarray}
\Pi \stackrel{\ast}{\to}  
\begin{cases}
(1+\rho_2)(e_1 +e_1^{*}) +  y_{12}^{1/2}(e_2 +e_2^{*}) \ \ \text{if}\ \ \lim_{p\to\infty}n_1/n_2=y_{12}\leq 1,\\ 
\label{eqn: sym0}\\
y_{12}^{-1/2}(1+\rho_2)(e_1 +e_1^{*}) + (e_2 +e_2^{*}) \ \ \text{if}\ \ \lim_{p\to\infty}n_1/n_2=y_{12}\geq 1.  
%\ \ 
%%\bigg(\lim \sqrt{\frac{\min(n_1,n_2)}{n_1}}\bigg)(1+\rho_2)(e_1 +e_1^{*}) +  \bigg(\lim \sqrt{\frac{\min(n_1,n_2)}{n_2}}\bigg)(e_2 +e_2^{*}) 
%\int_{\mathbb{R}^{n}} e^{it^{\prime}x} f(x) d\lambda(x)\  \ \text{if}\  \ \mu\  \ \text{has density} \  \ f (\cdot)\\
\end{cases}
\end{eqnarray}
Note that depending on whether $y_{12}=0$ or $\infty$, the second or the first term respectively drop out from the limit sums. 
 We can conclude that the LSD of $\Pi$ also exists almost surely and equals the probability law of the self-adjoint variable in (\ref{eqn: sym0})\end{Example}

%\begin{remark}
%Consider any symmetric polynomial 
%$$\Pi = \sum_{k=1}^{K} \gamma_k \sqrt{p^{-1}n_{\omega_k}}\bigg[\left(\prod_{j=1}^{J_k} C_{\alpha_{kj}}^{\eta_{kj}}\right)-\left(\prod_{j=1}^{J_k} \rho_{\alpha_{kj}}\right)\text{I}_p\bigg]$$ of $\{C_l,C_l^{*}:\ 1 \leq l \leq t\}$ for any  $\alpha_{kj} \in \{1,2,\ldots, t\}$, $\eta_{kj} \in \{1,*\}$, $\omega_k = {\rm{argmax}}\{n_{\alpha_{kj}}: 1 \leq j \leq J_k\}$, real $\gamma_k$ for all  $1 \leq j \leq J_k,\ 1 \leq k \leq K,\ K \geq 1 $.  Let  $\theta_{jk} = \prod_{\substack{j^\prime =1 \\ j^\prime \neq j}}^{J_k} \rho_{\alpha_{kj^\prime}}$, $\sigma_{uv}^2 = \lim n_{v}^{-1}n_u \in [0,\infty]$. Then by Theorem \ref{thm: 2}(b), the ESD of $\Pi$ converges weakly almost surely to the distribution of $\sum_{k=1}^{K} \gamma_k\left(\sum_{j=1}^{J_k}\theta_{kj} \sigma_{\omega_k\alpha_{kj}} e_{\alpha_{kj}}^{\eta_{kj}}\right)$.  This follows from arguments similar to those given in Example 2 of Chapter 7 in \cite{bose2018large}. ??what if some sigma's are infinity? Motivation for this example is not clear. What do we wish to say?  why not give a simple specific $\Pi$ like we did in our examples in the book to show which terms are negligible... and then indicate the general case..
%\end{remark}

\begin{proof}[Proof of Theorem \ref{thm: 2}] (a) We shall give the detailed proof only for the special case where all the $n_l$'s $\rho_l$'s are equal to say $n$ and $\rho$ respectively. For any $k \geq 1$ and $\epsilon_1,...,\epsilon_k\in\{1, *\}$, we will consider the limit of the following as $p,n\to\infty$ with $p/n\to 0$:
\begin{equation}\label{eq:tracecross}\dfrac{1}{p}\mathbb E\text{Tr}(E_{l_{1}}^{\epsilon_1}\cdots E_{l_{k}}^{\epsilon_k})=\dfrac{1}{p (np)^{k/2}}\hspace{-10pt}\sum_{\stackrel{1\leq i_1,\ldots,i_k\leq p}{1\leq j_1,\ldots, j_k\leq n}}
\hspace{-15pt}\mathbb E\big[\prod_{t=1}^k\big(a_{l_{t}i_tj_t}a_{l_{t}i_{t+1}j_t}-\rho\delta_{i_ti_{t+1}}\big)\big]
\end{equation}
 with the understanding that $i_{k+1}=i_1$, and as ordered pairs, for all $1\leq r \leq k$,  
\begin{eqnarray}  
(a_{l_{r}i_rj_r}, a_{l_{r}i_{r+1}j_r}) = \begin{cases} 
(x_{i_rj_r}^{(l_r)}, y_{i_{r+1}j_{r}}^{(l_r)}) \ \ \text{if}\ \  \epsilon_r=1,\nonumber\\
\\
(y_{i_rj_r}^{(l_r)}, x_{i_{r+1}j_r}^{(l_r)})\ \ \text{if}\ \ \epsilon_r=*.\label{eq:axy}
\end{cases}
\end{eqnarray}
Consider the following collection of all \textit{ordered} pairs of indices that appear in the above formula: 
$$P=\{(i_r, j_r), (i_{r+1}, j_r), 1\leq r\leq k\}.$$
(i) Suppose there is a  pair say,  $(i_r, j_r)\in P$ that appears only once. %such $(i_r, j_r)\neq (i_s, j_s)$ for all $s\neq r$.  
Then  $(i_r, j_r)\neq (i_{r+1}, j_r)$ and hence  $i_r\neq i_{r+1}$. 
As a consequence, the variable  
$a_{l_{r}i_rj_r}$ is independent of all other variables and we get %in the product 
\begin{eqnarray}\mathbb E\big[\prod_{t=1}^k\big(a_{l_{t}i_tj_t}a_{l_{t}i_{t+1}j_t}-\rho\delta_{i_ti_{t+1}}\big)\big]
&=&\hspace{-5pt}\mathbb E[a_{l_ri_rj_r}] \mathbb E\Big[a_{l_ri_{r+1}j_r} %\nonumber \\
%&&\hspace{-5pt}
\prod_{t\neq r}\big(a_{l_{t}i_tj_t}a_{l_{t}i_{t+1}j_t}-\rho\delta_{i_ti_{t+1}}\big)\Big]%\nonumber
%\\
%&=&
=0. \hspace{1 cm}\label{eq:factor0}
\end{eqnarray}
The same conclusion holds if a pair $(i_{r+1},j_r)$ occurs only once in $P$. So we can restrict attention to the subset of $P$ where each pair is repeated at least twice, and we continue to call this reduced subset by  $P$. \vskip3pt

\noindent (ii) Suppose in $P$ there is a $j_r$ such $j_r\neq j_s$ for all $s\neq r$, then the pair
$(a_{l_{r}i_rj_r},a_{l_{r}i_{r+1}j_r})$ is independent of all other factors in the product. Hence, 
\begin{eqnarray*}\mathbb E\big[\prod_{t=1}^k\big(a_{l_{t}i_tj_t}a_{l_{t}i_{t+1}j_t}-\rho\delta_{i_ti_{t+1}}\big)\big]&\hspace{-7pt}=&\hspace{-7pt}\underbrace{\mathbb E[a_{l_ri_rj_r}a_{l_ri_{r+1}j_r}-\rho\delta_{i_ri_{r+1}}]}_{=0}%\times\\
%&&\hspace{-16pt}
\mathbb E\big[\prod_{t\neq r}\big(a_{l_{t}i_tj_t}a_{l_{t}i_{t+1}j_t}-\rho\delta_{i_ti_{t+1}}\big)\big] = 0.%\\
%&=&0.
\end{eqnarray*}Hence we restrict attention to the subset of $P$ where each $j_r$ occurs in at least four pairs i.e. in $(i_r,j_r),(i_{r+1},j_r)$ and also in $(i_s,j_s),(i_{s+1},j_s)$ for some $s\neq r$. 
We  continue to call this reduced subset by $P$, and the corresponding pairs, edges. If $j_r=j_s$ then they are said to be matched and likewise for the $i$-vertices. 

Define the set of vertices $V_{I}$ and $V_J$ which are the distinct indices from  $\{i_1,...,i_k\}$ and $\{j_1,...,j_k\}$ respectively. Note that there are at most $2k$ edges in $P$ but each edge appears at least twice. Let $E$ be the set of distinct edges between the vertices in $V_I$ and $V_J$. This defines a simple connected bi-partite graph. 
%Consider the set $P$ as a  simple connected bipartite graph with vertex set $V=V_{I}\cup V_{J}$, where $V_I=\{i_1,...,i_k\}$ and $V_J=\{j_1,...,j_k\}$, and with edge set $E$ exactly equal to $P$. 
 Then clearly, $\#E\leq k$.
Since every $j$-index was originally matched, $\#V_J\leq k/2.$
We also know from the connectedness property that
\begin{equation}\label{eq:connectedgraph}\#V_I+\#V_J \leq \#E+1\leq k+1.
\end{equation}
Hence the contribution to (\ref{eq:tracecross}) is bounded above by 
\begin{equation}\label{order}
    O\left(\dfrac{p^{\#V_I}n^{\#V_J}}{p^{k/2+1}n^{k/2}}\right)=O\left(\dfrac{p^{k+1-\#V_J} n^{\#V_J}}{p^{k/2+1}n^{k/2}}\right)=O\left(\left(\dfrac{p}{n}\right)^{k/2-\#V_J}\right).
\end{equation}
If $\#V_J<\dfrac{k}{2}$ then the above expression goes to $0$. 
So the only possible non-zero contribution to the limit of (\ref{eq:tracecross}) will come when $\#V_J=k/2$. This immediately shows that if $k$ is odd, then we do not get such a contribution and hence, $$\dfrac{1}{p}\mathbb E \text{Tr}(E_{l_{1}}^{\epsilon_1} \cdots E_{l_{k}}^{\epsilon_k})\to 0.$$
So now consider the (contributing) case when $k$ is even and $\#V_J=k/2$, that is, let $k=2m$ and $\#V_J=m$. 
Then $$O\left(\dfrac{p^{\#V_I}n^{\#V_J}}{p^{k/2+1}n^{k/2}}\right)=O(p^{\#V_I-(k/2+1)})=O(p^{\#V_I-(m+1)}).$$ On the other hand, from (\ref{eq:connectedgraph}), when $k=2m$ and $\#V_J=m$, we get $\#V_I\leq m+1$. So for a possible non-zero contribution, we must have $\#V_I=m+1$. 
This implies that 
$$m+1+m=\#V_I+\#V_J\leq \#E+1\leq 2m+1$$ and hence $\#E=2m$. In other words, each edge must appear exactly $2$ times.

Suppose $(i_r,j_r)=(i_{r+1},j_r)$ for some $r$. Since each edge appears exactly twice, this pair will be independent of all others and therefore 
\begin{eqnarray*}\mathbb E\big[\prod_{t=1}^k\big(a_{l_{t}i_tj_t}a_{l_{t}i_{t+1}j_t}-\rho\delta_{i_ti_{t+1}}\big)\big]&\hspace{-5pt}=&\hspace{-5pt}\underbrace{\mathbb E[a_{l_ri_rj_r}a_{l_ri_{r+1}j_r}-\rho\delta_{i_ri_{r+1}}]}_{=0} %\times\\
%&&\hspace{-15pt}
\mathbb E\big[\prod_{t\neq r}\big(a_{l_{t}i_tj_t}a_{l_{t}i_{t+1}j_t}-\rho\delta_{i_ti_{t+1}}\big)\big] = 0. %%\\
%&=&0.
\end{eqnarray*}
Hence, such a combination  cannot contribute to (\ref{eq:tracecross}). So we may assume from now on that for every $r$, $i_r\neq i_{r+1}$ and hence $\delta_{i_ri_{r+1}}=0$ always. We 
continue to call this reduced subset by $P$. As a consequence,  (\ref{eq:tracecross}) reduces to 
$$\dfrac{1}{p}\mathbb E\text{Tr}(E_{l_1}^{\epsilon_1}\cdots E_{l_k}^{\epsilon_k})=\dfrac{1}{p^{m+1}n^m}\sum_{P}\mathbb E\big(\prod_{r=1}^{2m} a_{l_{r}i_rj_r}a_{l_{r}i_{r+1}j_{r}}\big).$$
Due to the preceding discussion, we have two situations. Suppose $(i_r,j_r)=(i_s,j_s)$ for some $s\neq r$. Note that $j_r$ and $j_s$ are also adjacent to $i_{r+1}$ and $i_{s+1}$ respectively. Due to the nature of the edge set $P$, 
this forces $i_{r+1}=i_{s+1}$. Similarly if $(i_r, j_r)=(i_{s+1}, j_s)$ then it would force $i_{r+1}=i_s$.
Let $\mathcal P(2m)$ be the set of all possible pair partitions of the set $\{1,...,2m\}$. We will think of each block in the partition to represent the equal pairs of edges in the graph. Now recalling the definition  (\ref{eq:axy}), the moment structure, and the above developments, it is easy to verify that the possibly contributing part of the above sum (and hence of (\ref{eq:tracecross})) can be re-expressed as  
\begin{eqnarray}
\sum_{\stackrel{i_1,...,i_{2m}}{j_1,...,j_{2m}}}\hspace{-2pt}\sum_{\pi\in \mathcal{P}_{2}(2m)}\hspace{-3pt}\dfrac{1}{p^{m+1}n^m}\hspace{-7pt}\prod_{(r,s)\in\pi}\hspace{-9pt}\Big[\delta_{l_rl_s}\delta_{\epsilon_r\epsilon_s}(\rho^2\delta_{i_ri_{s+1}}\delta_{j_rj_s}\delta_{i_{r+1}i_{s}}%\\
+\delta_{i_ri_s}\delta_{j_rj_s}\delta_{i_{r+1}i_{s+1}})\nonumber\\ 
+\delta_{l_rl_s}(1-\delta_{\epsilon_r\epsilon_s})(\rho^2\delta_{i_ri_s}\delta_{j_rj_s}\delta_{i_{r+1}i_{s+1}}+\delta_{j_rj_s}\delta_{i_ri_{s+1}}\delta_{i_{r+1}i_s})\Big]\label{eq:rho_one}
\end{eqnarray}
which equals 
\begin{eqnarray}
&&\hspace{-20pt}\sum_{\stackrel{i_1,...,i_{2m}}{j_1,...,j_{2m}}}\sum_{\pi\in P_2(2m)}\dfrac{1}{p^{m+1}n^m}\prod_{(r,s)\in\pi}\delta_{l_rl_s}\delta_{j_rj_s}\Big[\delta_{i_ri_{s+1}}\delta_{i_si_{r+1}}\big(\rho^2\delta_{\epsilon_r\epsilon_s} %\nonumber\\
%&&
+(1-\delta_{\epsilon_r\epsilon_s})\big)+\text{other terms}\Big].\hspace{1 cm}\label{red_eqn}
\end{eqnarray}
We first consider a special case and extract some crucial information that will be useful for the general case. Suppose $t=1,\rho=1$. Then $Y=X$, and we can take 
$\epsilon_r$ to be same for all $r, s$. We know that $\sqrt{np^{-1}}(XX^{*}-I_p)$ converges to a semi-circular variable. This immediately implies that the limit of (\ref{eq:tracecross}) and hence of 
(\ref{eq:rho_one}) and (\ref{red_eqn}) in this special case equals $\#{\rm{NC}}_2(2m)=C_m$, ;the $m$th Catalan number. This means that 
$$\sum_{i_1,i_2,\ldots, i_m=1}^{p}\sum_{j_1,j_2,\ldots,j_m=1}^{n}\sum_{\pi\in P_2(2m)}\dfrac{1}{p^{m+1}n^m}\hspace{-5pt}\prod_{(r,s)\in\pi}\hspace{-5pt}\delta_{j_rj_s}(\delta_{i_ri_{s+1}}\delta_{i_si_{r+1}}+\delta_{i_ri_s}\delta_{i_{r+1}i_{s+1}}) \to \hspace{-10pt}\sum_{\pi\in {\rm{NC}}_2(2m)}\hspace{-10pt}1.$$

Now we note that for $\pi\in {\rm{NC}}_2(2m)$, $$\sum_{i_1,...,i_{2m},j_1,...,j_{2m}}\dfrac{1}{p^{m+1}n^m}\prod_{(r,s)\in\pi}\delta_{j_rj_s}\delta_{i_ri_{s+1}}\delta_{i_si_{r+1}}=1.$$

The reason is as follows. If $(r,s)\in\pi$ then $j_r=j_s$, and there are $n^m$ ways of choosing the $j-$indices. Now let $\gamma=(1,2,...,2m)$ be the cyclic permutation. Then we note $$\prod_{(r,s)\in\pi}\delta_{i_ri_{s+1}}\delta_{i_si_{r+1}}=\prod_{(r,s)\in\pi}\delta_{i_ri_{\gamma\pi(r)}}\delta_{i_si_{\gamma\pi(s)}}=\prod_{r=1}^{2m}\delta_{i_ri_{\gamma\pi(r)}}.$$
So $i_r=i_{\gamma\pi(r)}$ for each $r$ if the above product is 1. But this means $i$ is constant on each block of $\gamma\pi$. As $\pi\in {\rm{NC}}_2(2m)$, $|\gamma\pi|=m+1$ and thus there are $p^{m+1}$ choices in total for the $i$'s.

These arguments establish that 
\begin{eqnarray}
\lim_{n\to\infty}\hspace{-5pt}\sum_{\pi\in P_2(2m)}\sum_{\stackrel{i_1,...,i_{2m}}{j_1,...,j_{2m}}}\hspace{-5pt}\dfrac{1}{p^{m+1}n^m}\hspace{-5pt}\prod_{(r,s)\in\pi}\hspace{-5pt}\delta_{j_rj_s}(\delta_{i_ri_{s+1}}\delta_{i_si_{r+1}}+\delta_{i_ri_s}\delta_{i_{r+1}i_{s+1}})\label{eq:addup}\\\ =\lim \sum_{\pi\in {\rm{NC}}_2(2m)}\dfrac{1}{p^{m+1}n^m}\prod_{(r,s)\in\pi}\delta_{j_rj_s}\delta_{i_ri_{s+1}}\delta_{i_si_{r+1}}.
\end{eqnarray}
This implies that the rest of the terms in (\ref{eq:addup}) must go to $0$, since these quantities are all non-negative. 

Now we consider the general case where $t\geq 1$ but for the moment assume that $\rho_l$'s are equal but  the common value is not necessarily equal to $1$. Going back to (\ref{red_eqn}) and noting that for each $\pi\in P_2(2m)$ and for each $(r,s)\in \pi$, $$|\delta_{l_rl_s}\delta_{j_rj_s}(\delta_{i_ri_{s+1}}\delta_{i_si_{r+1}}(\rho^2\delta_{\epsilon_r\epsilon_s}+(1-\delta_{\epsilon_r\epsilon_s}))+\text{other terms})|\leq 1,$$ we conclude that the expression in (\ref{red_eqn}) converges to 
\begin{equation*}
\sum_{\pi\in {\rm{NC}}_2(2m)}\prod_{(r,s)\in\pi}\delta_{l_{r}l_{s}}(\rho^2\delta_{\epsilon_r\epsilon_s}+(1-\delta_{\epsilon_r\epsilon_s}))\ \ \text{as}\ \ n \to \infty.
\end{equation*}

Now note that $\rho^2\delta_{\epsilon_r\epsilon_s}+(1-\delta_{\epsilon_r\epsilon_s})=(\rho^2)^{\delta_{\epsilon_r\epsilon_s}}$. Let $T(\pi)=\#\{(r,s)\in\pi:\epsilon_r=\epsilon_s\}$. Then the limit can be re-expressed as $$\sum_{\pi\in {\rm{NC}}_2(2m)}\prod_{(r,s)\in \pi}(\rho^2)^{\delta_{\epsilon_r\epsilon_s}}\delta_{l_rl_s}=\sum_{\pi\in {\rm{NC}}_2(2m)}\prod_{(r,s)\in \pi}\delta_{l_rl_s}(\rho^2)^{T(\pi)}.$$
If $e_1,...,e_{2m}$ are freely independent elliptic elements each with parameter $\rho^2$ in some NCP $(\mathcal{A}, \varphi)$, then the above expression  is nothing but $\varphi(e_{l_1}^{\epsilon_1}...e_{l_{2m}}^{\epsilon_{2m}})$. This proves the $*$-convergence (for the special case where all $\rho_l$'s  and all
$n_l$'s are identical.

In particular, if $\rho_l=1$ for all $l$, then $\{E^{(l)}, 1\leq i \leq t\}$ are asymptotically free semi-circular variables and if $\rho_l=0$ for all $l$, then they are asymptotically free circular variables. 

If we follow the above proof carefully, then it is clear that when we have possibly different $\rho_l$'s, and $n_l$'s, the argument for negligibility of the terms remains valid. Once we make the allowance for different $\rho$, the rest of the proof carries through and the product of $\{\rho_{l}^{T_{l}(\pi)}\}$ emerges in the limit. This completes the proof of the first part of (a). 

As discussed in the proof of Theorem \ref{thm: 1}, convergence with respect to the state $\tilde{\varphi}_p$ follows by the Borel-Cantelli Lemma after it is established that 
\begin{equation}\label{eq:fourthmom}\mathbb{E}\Big[p^{-1}{\rm Trace}(\Pi^k)-p^{-1}E [{\rm Trace}(\Pi^k)]\Big]^4=O(p^{-2}).
\end{equation}
%for any finite degree monomial $\Pi$ of $\{E_l,E_l^{*}:\ 1 \leq l \leq t\}$.  
Proof of (\ref{eq:fourthmom}) proceeds along lines similar to that in the proof of the first part of (a). We omit the details.  See the proof of Theorem 3.5 in \cite{BB2016} for similar arguments.  This complete the proof of (a).

\noindent (b) This is also similar to the proof of Theorem \ref{thm: 1}(b). Let the finite degree polynomial $\Pi$ of $\{E_l,E_l^{*}:\ 1 \leq l \leq t\}$ be symmetric. By Theorem \ref{thm: 2}(a), all moments of $\Pi$ converge almost surely. Also there is a $C > 0$, depending on $\Pi$ such that the limiting $k$th moment is bounded by $C^{k}$ for all $k$. This implies that these moments define a unique probability law whose support is a subset of the interval $[-C, \ C]$, and hence the ESD of $\Pi$ converges weakly almost surely to this law. 
\end{proof}

Figure \ref{fig2} reports the simulation results for a few polynomials when $p/n$ is small ($0.03$) for different values of $\rho$.

\bibliographystyle{abbrvnat}

%\bibliography{crosscov}

\begin{figure}[h!]
\includegraphics[height=70mm, width=50mm]{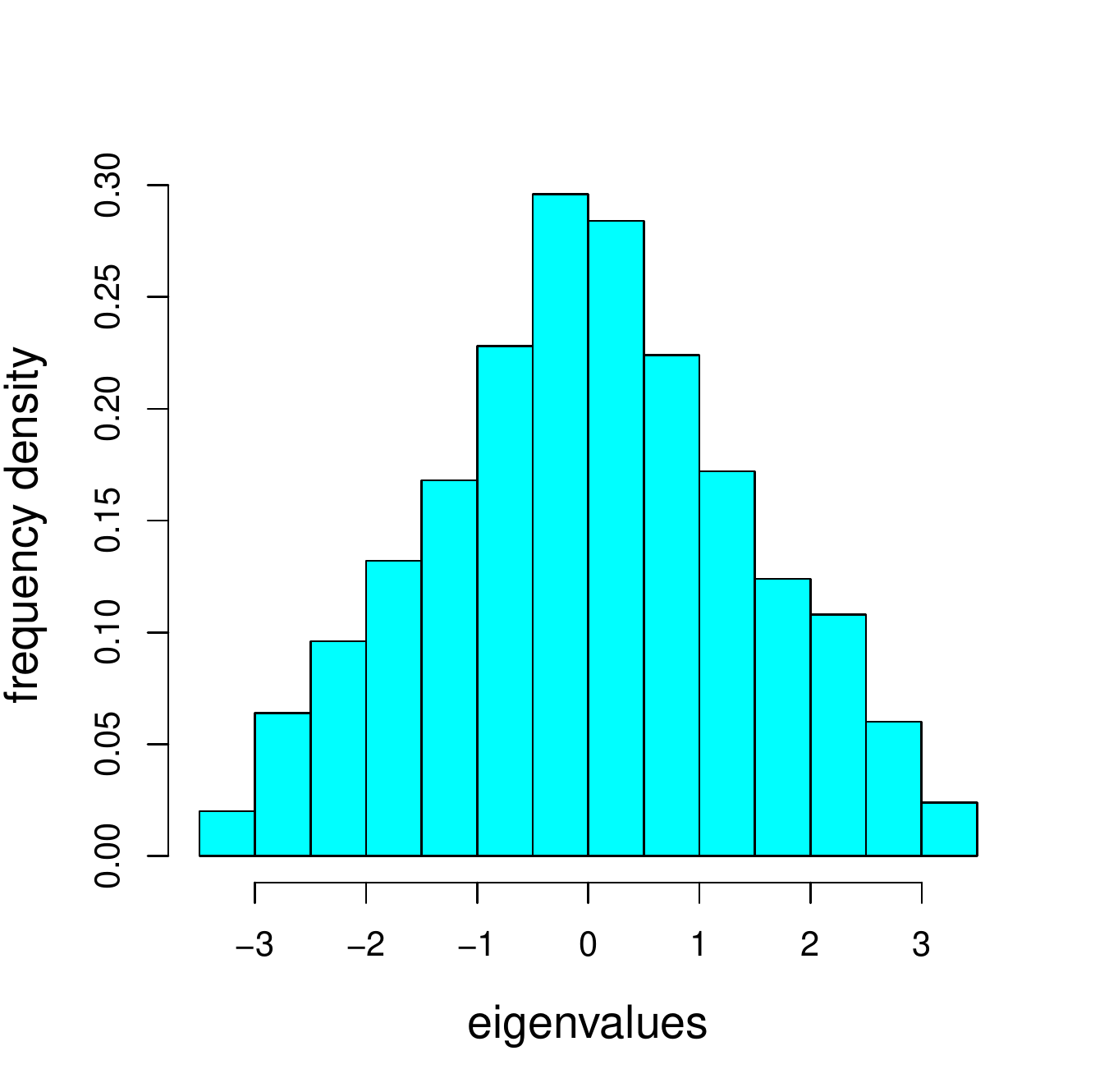}
\includegraphics[height=70mm, width=50mm]{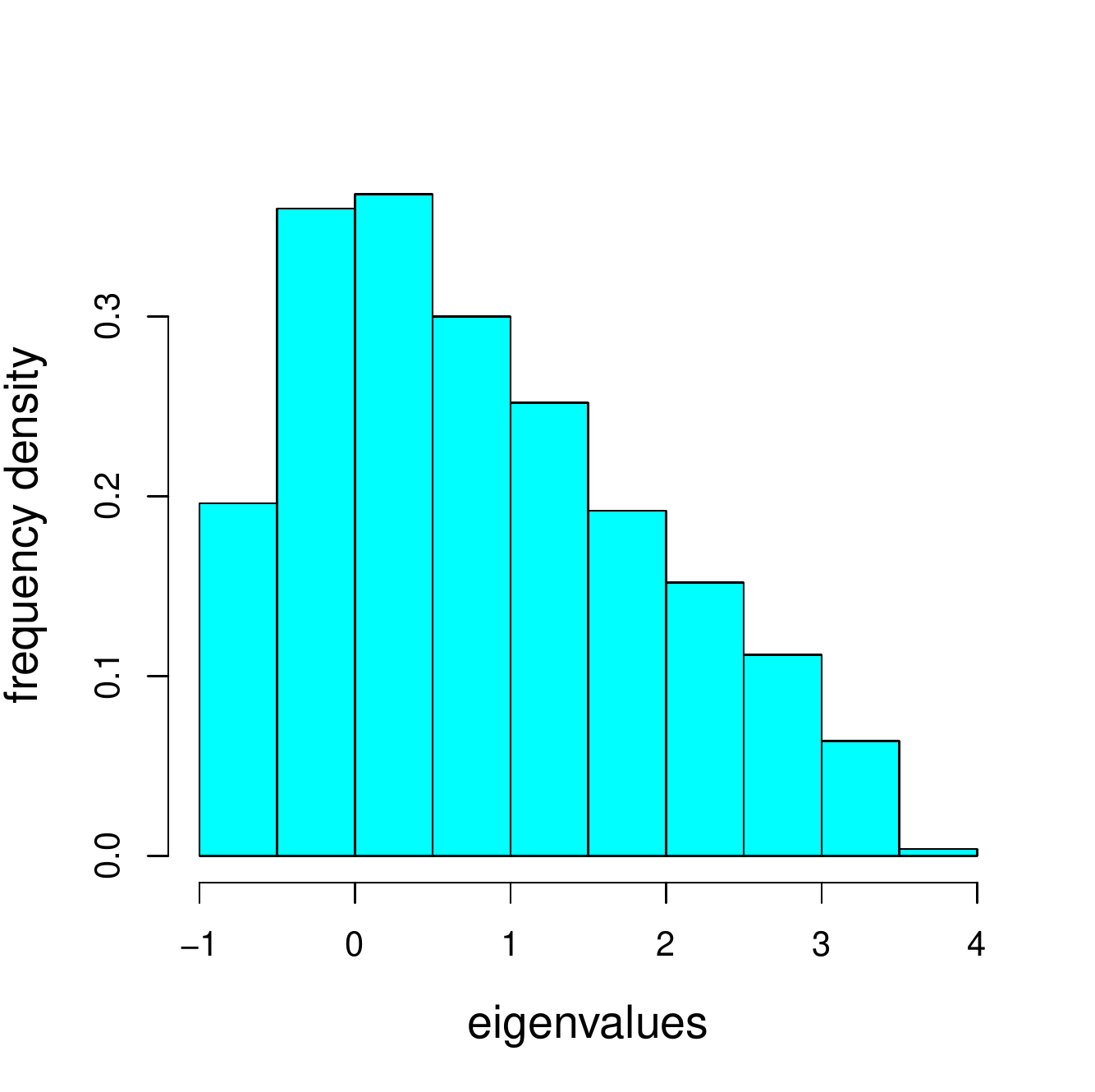}
\includegraphics[height=70mm, width=50mm]{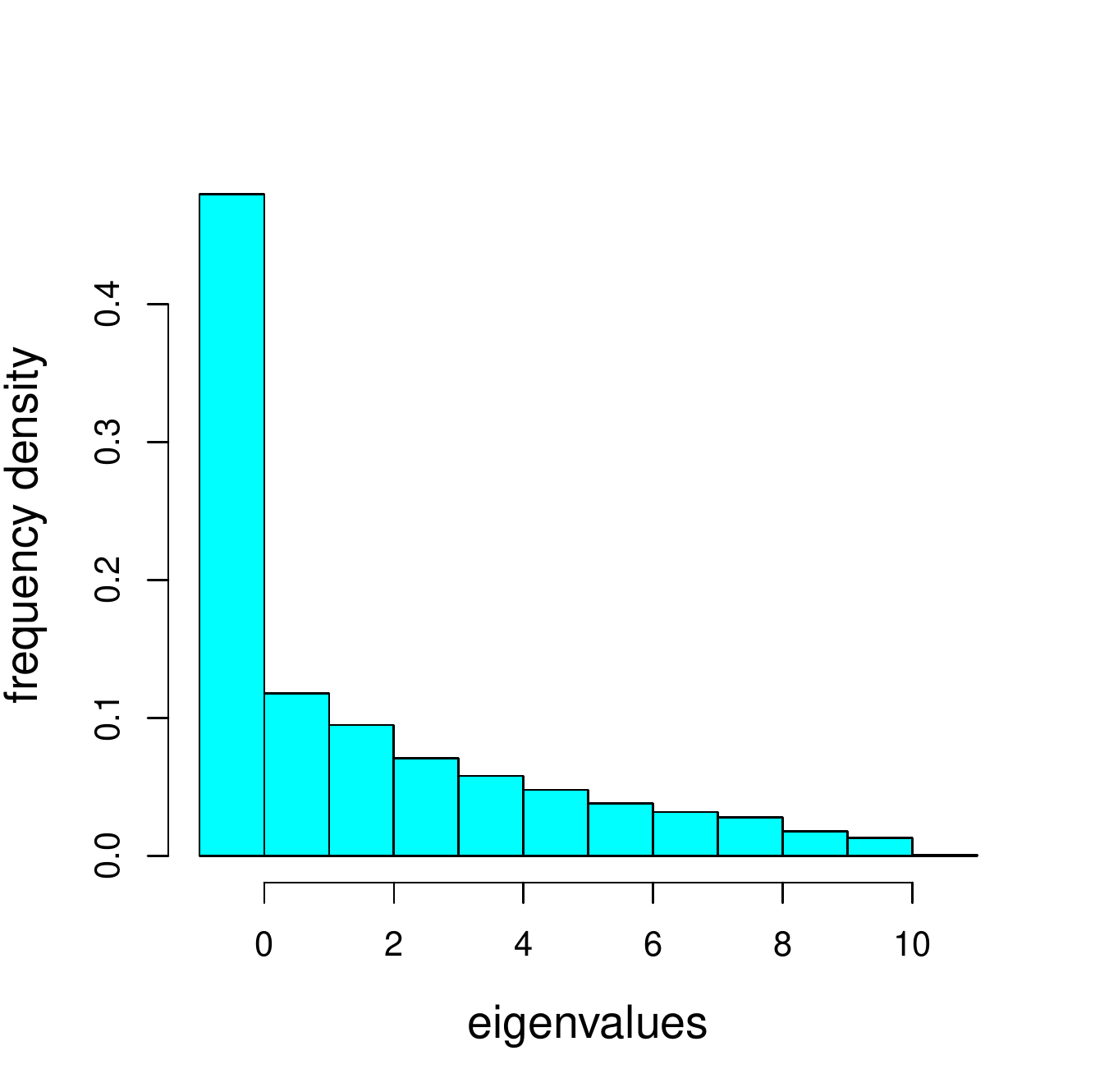}
\includegraphics[height=50mm, width=50mm]{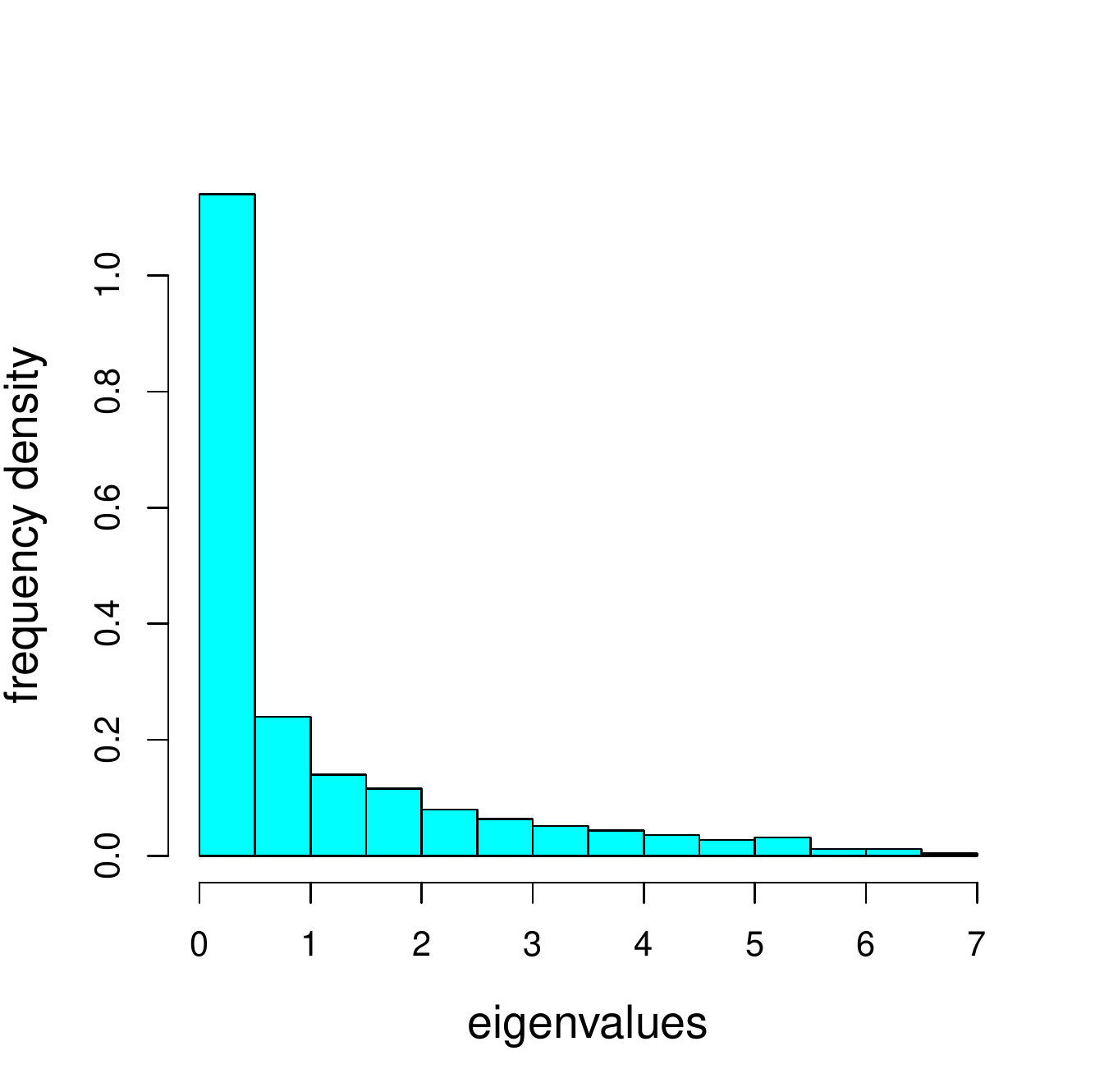}
\includegraphics[height=50mm, width=50mm]{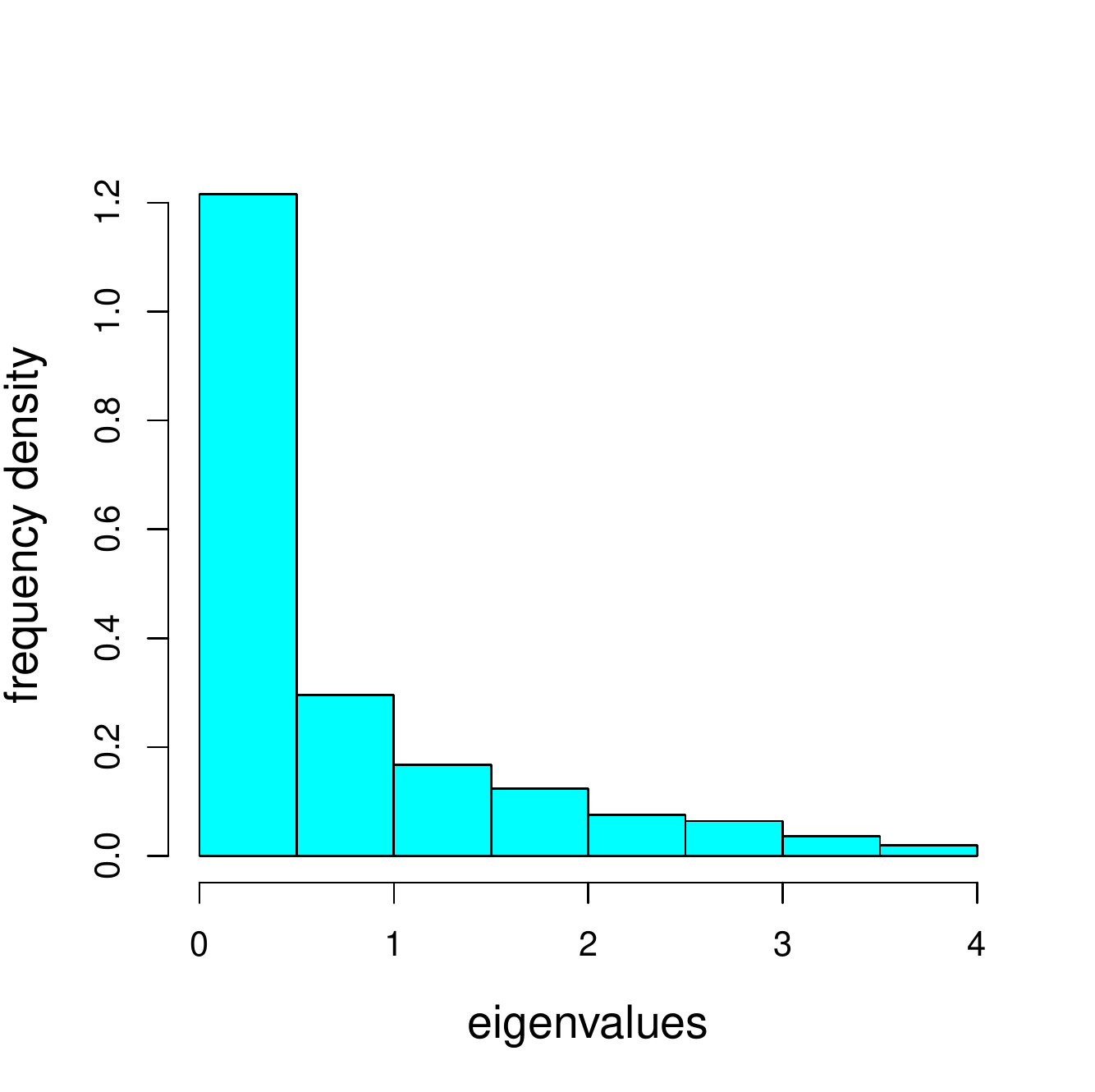}
\includegraphics[height=50mm, width=50mm]{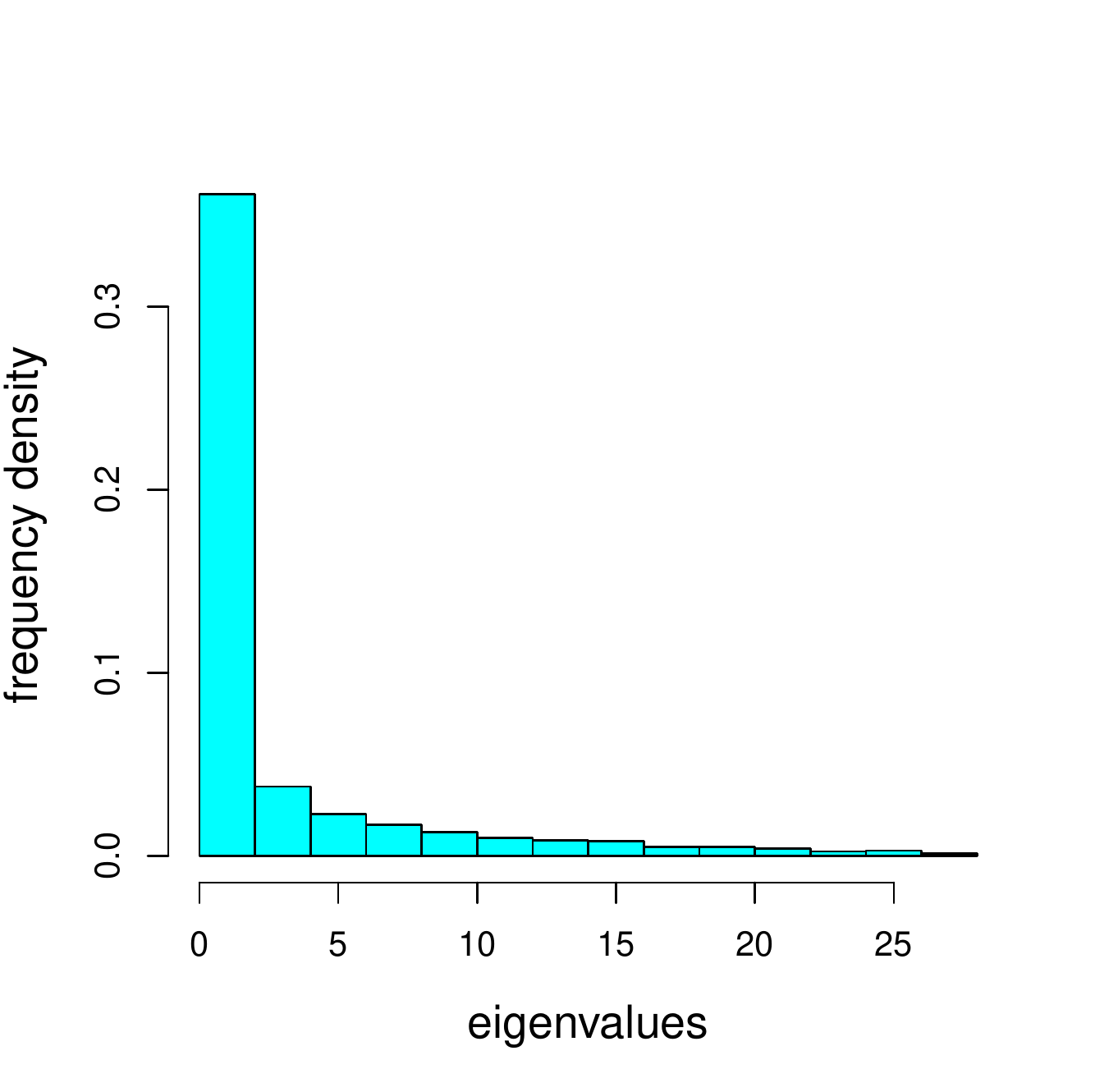}
\includegraphics[height=50mm, width=50mm]{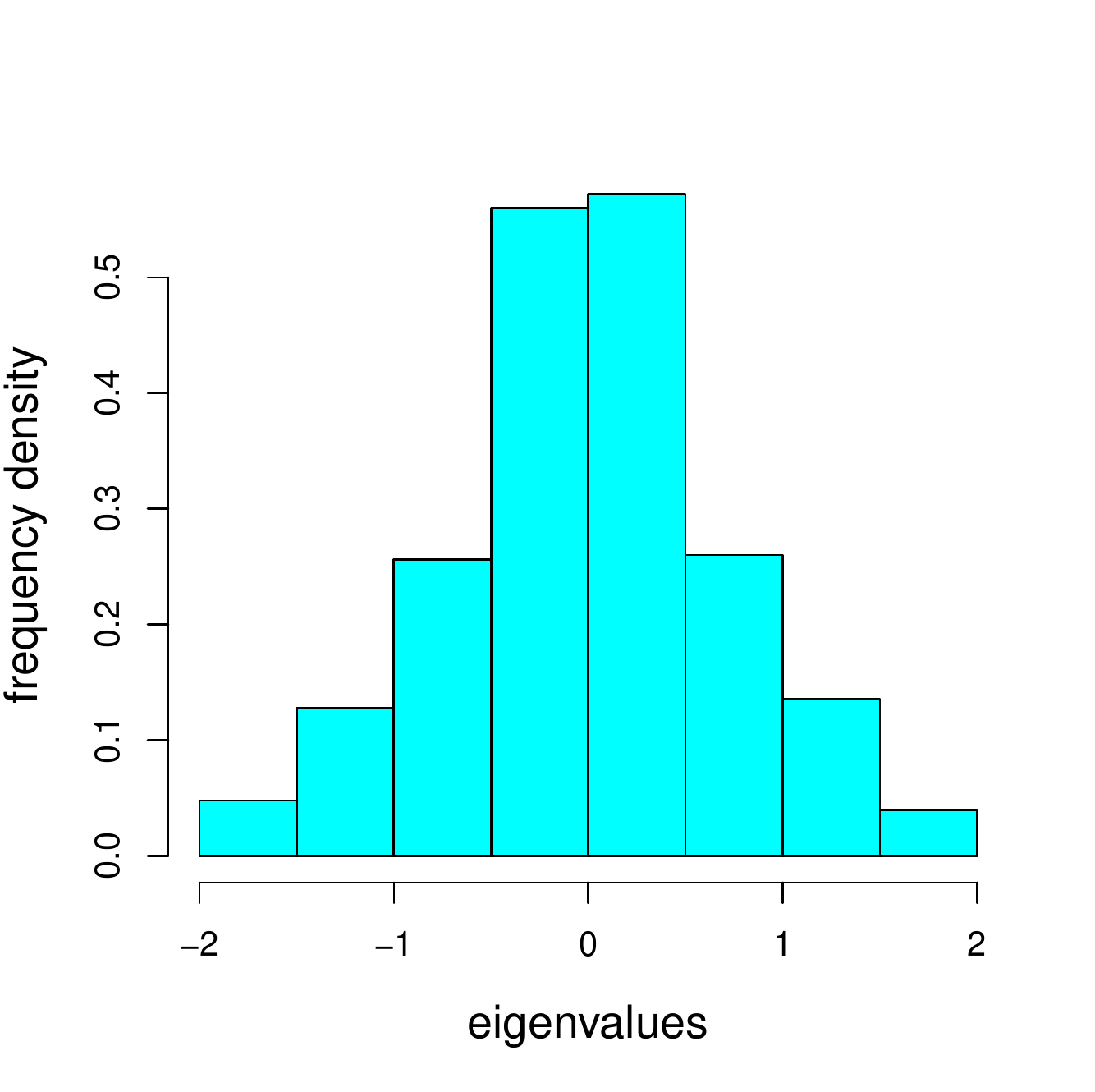}
\includegraphics[height=50mm, width=50mm]{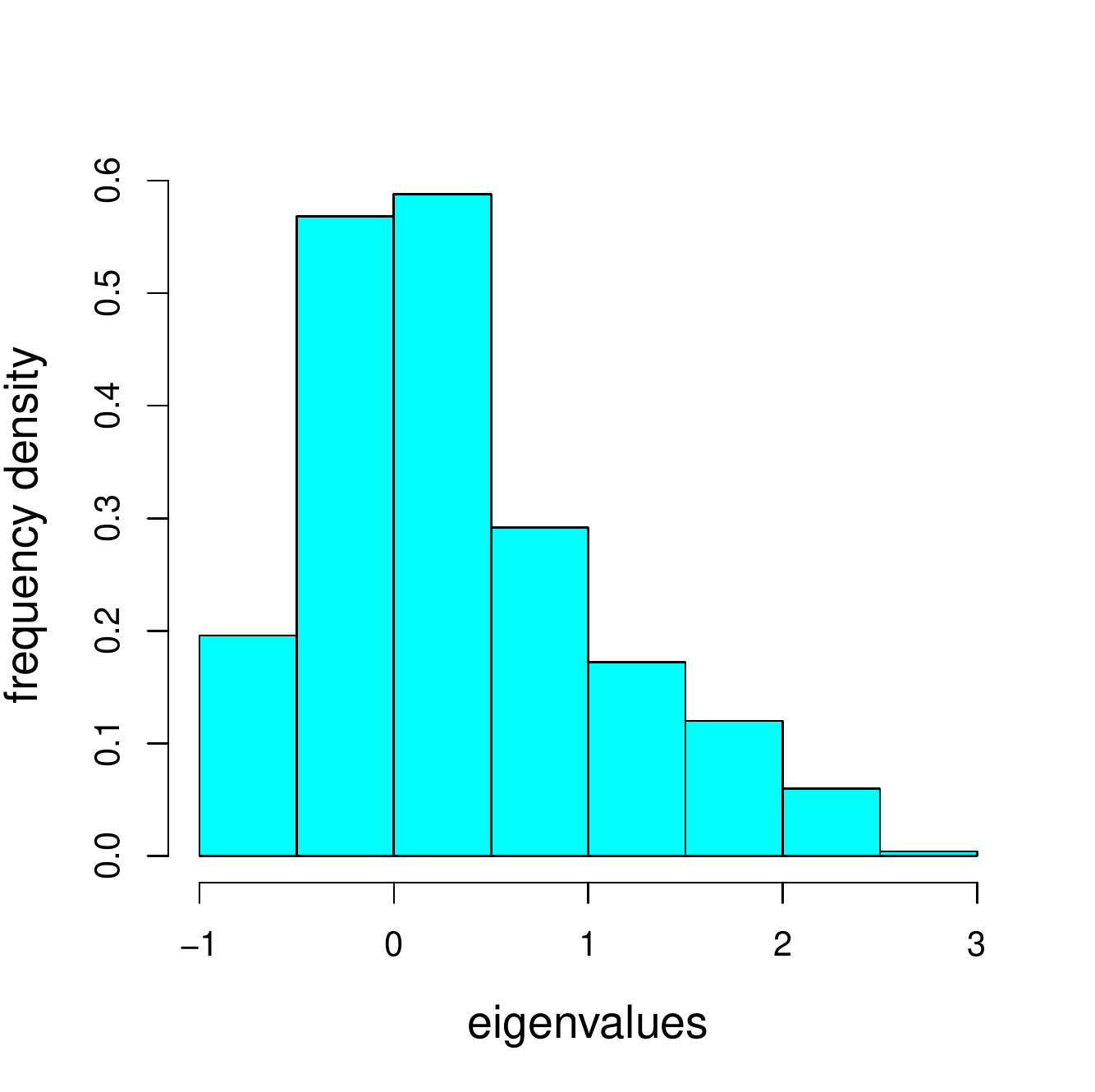}
\includegraphics[height=50mm, width=50mm]{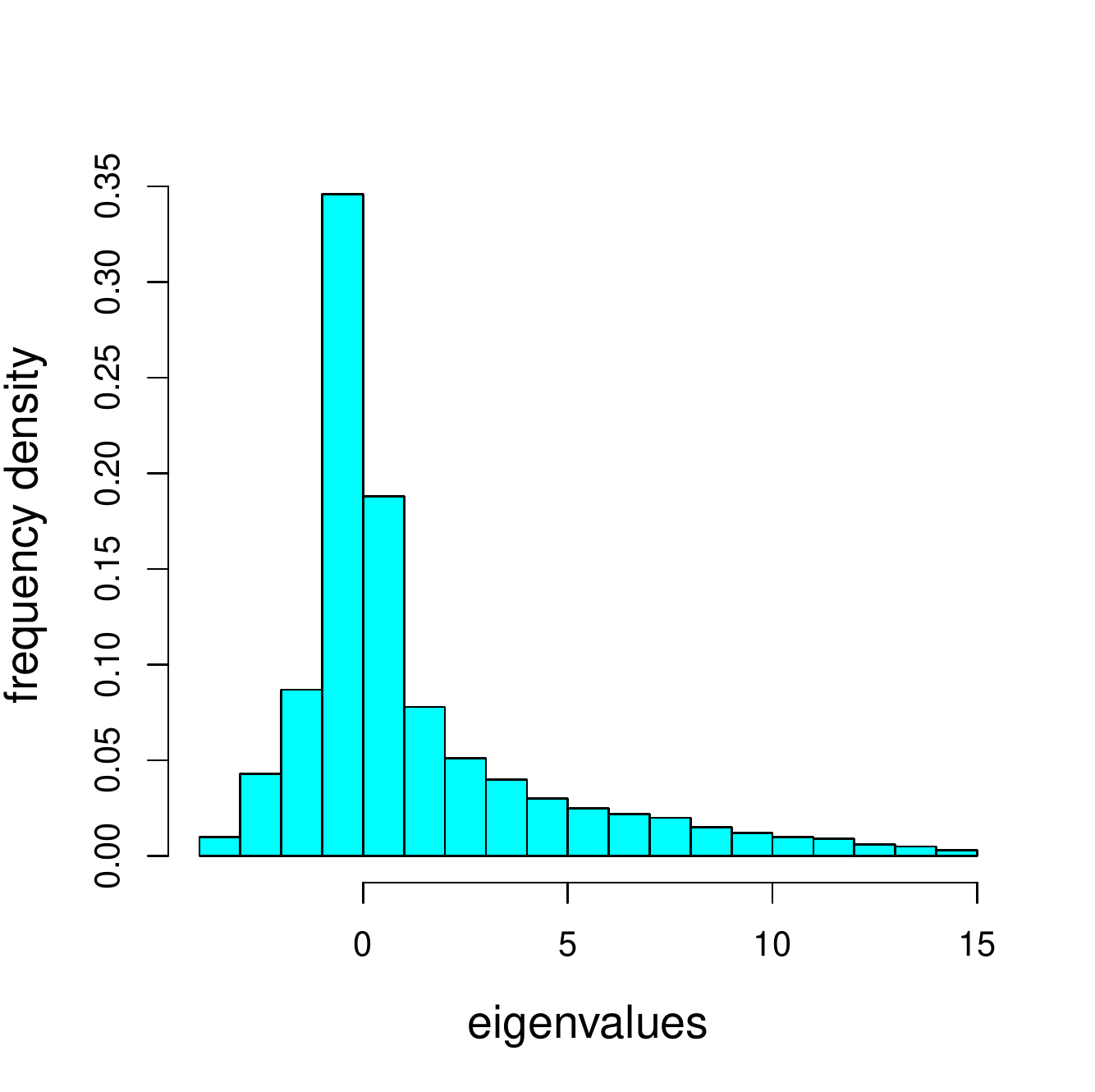}
\includegraphics[height=50mm, width=50mm]{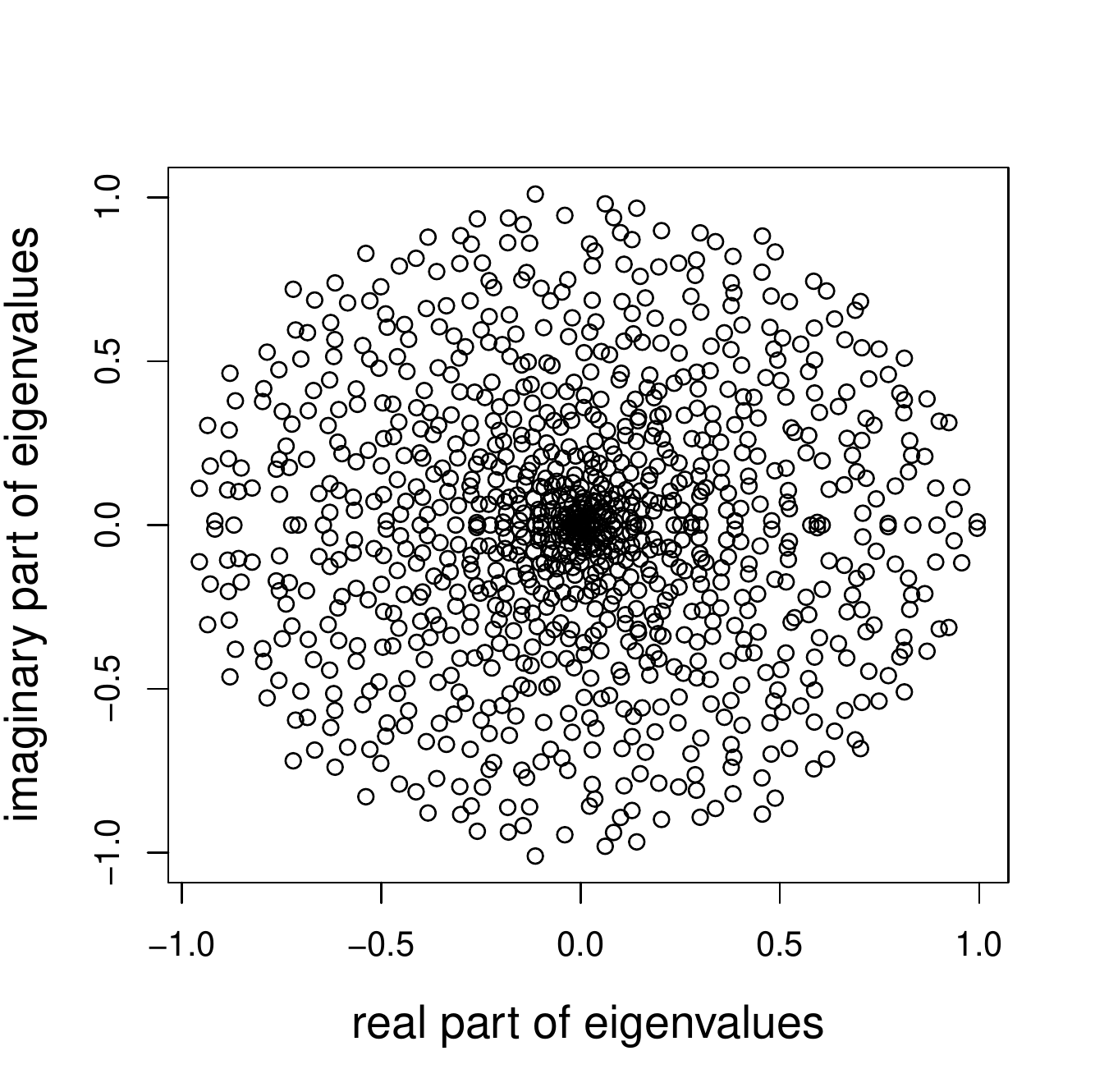}
\includegraphics[height=50mm, width=50mm]{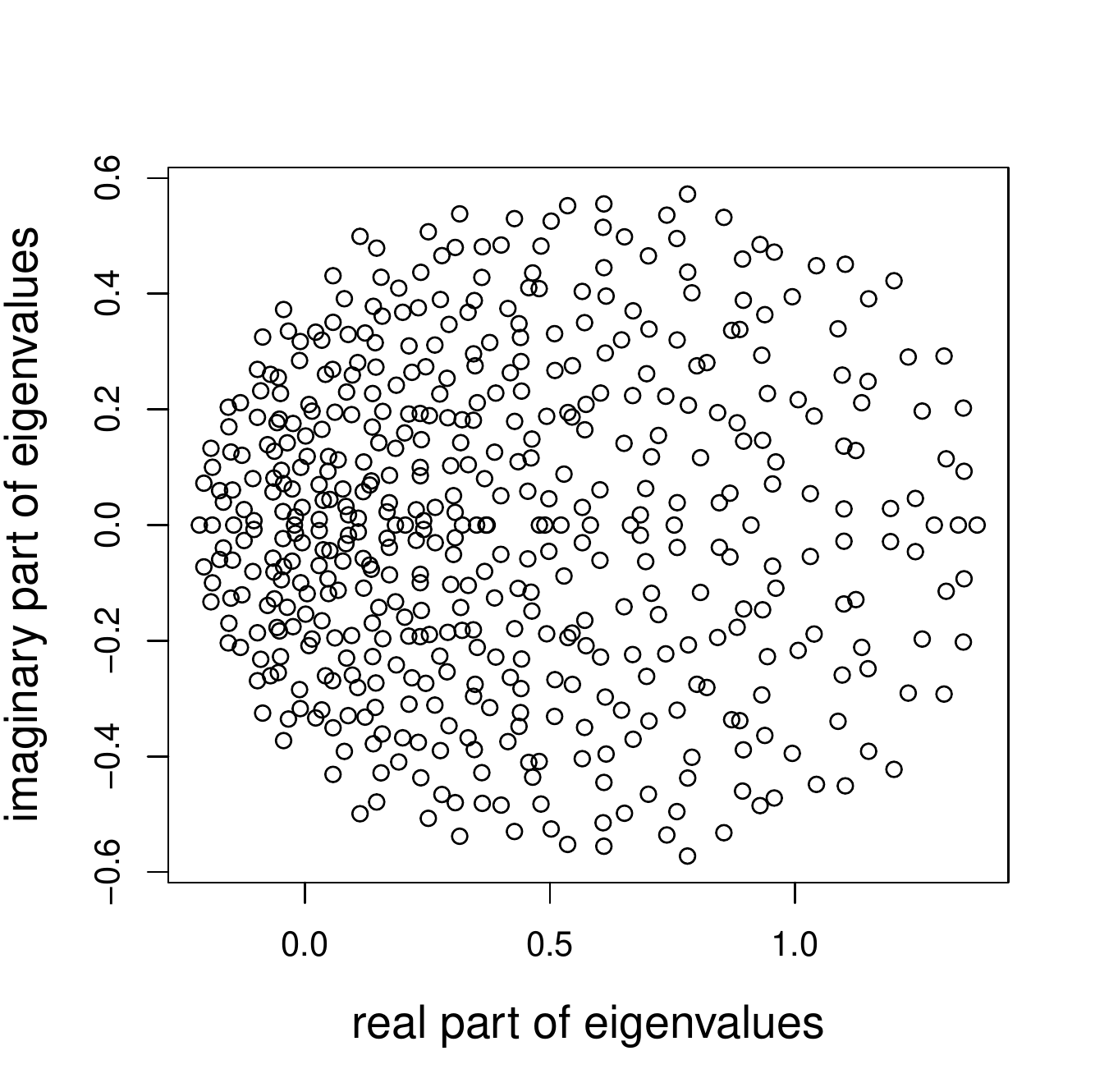}
\includegraphics[height=50mm, width=50mm]{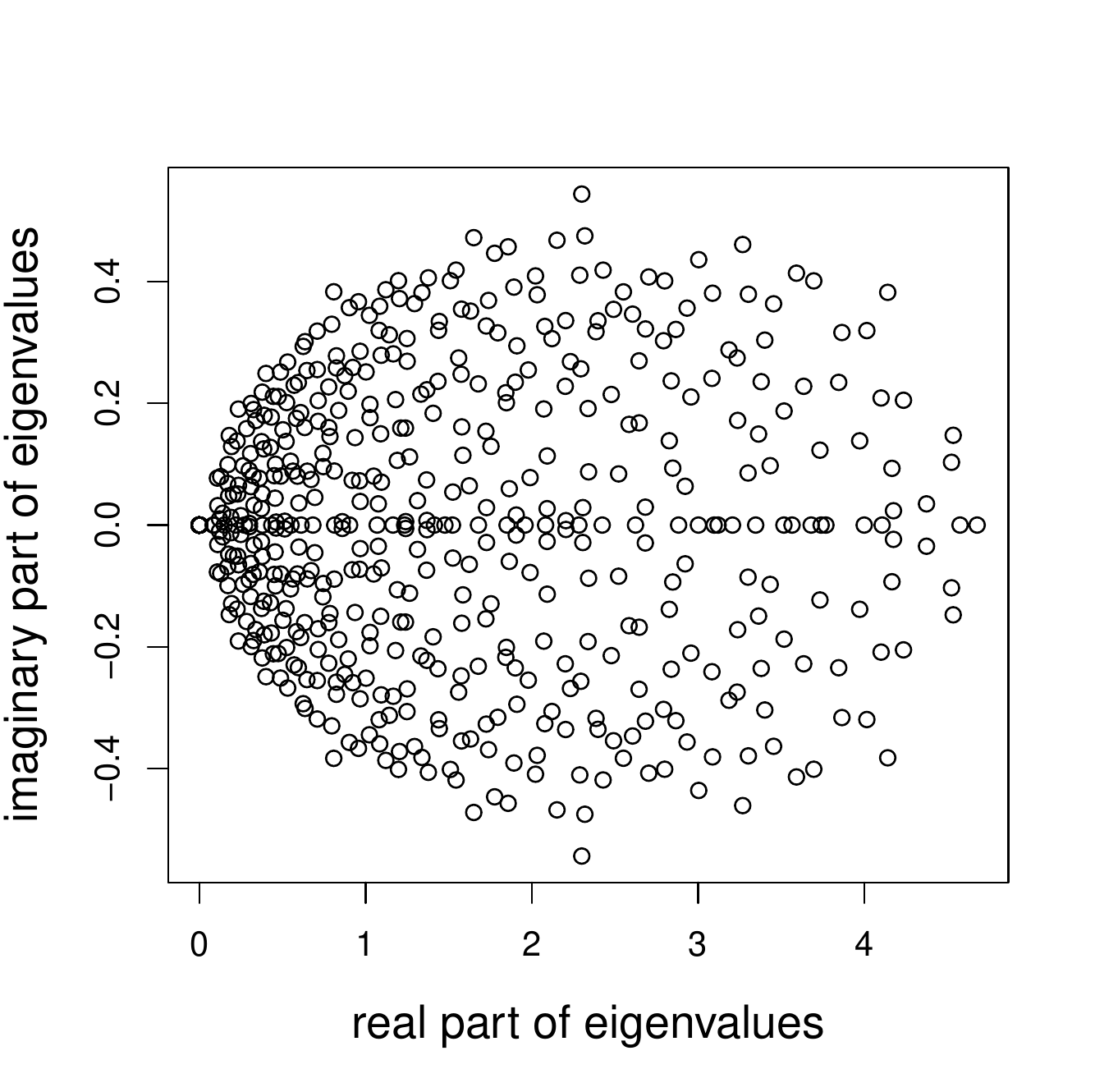}
\caption{Histogram for ESD of $(C+C^{*})$, $CC^{*}$ and $C_1C_2^{*}+C_2C_1^{*}$ ($n_1=n$, $n_2 = 2n$, $\rho_1 = \rho_2=\rho$) in Rows 1-3 respectively.  ESD of $C$ in Row 4. Column 1: $n=p=500$, $\rho=0$, 
Column 2: $n=1000, \ p=500$, $\rho = 0.4$ Column 3: $n=500, \ p=1000$, $\rho = 0.8$.}
%Row 1, columns 1: $(C+C^{*})$, Row 1, column 2: $CC^{*}$, Row 2, column 1: $C_1C_2^{*}+C_2C_1^{*}$. All figures are for $\rho=0$ and $y=1$.}
\label{fig1}\end{figure}

\begin{figure}
\includegraphics[height=55mm, width=50mm]{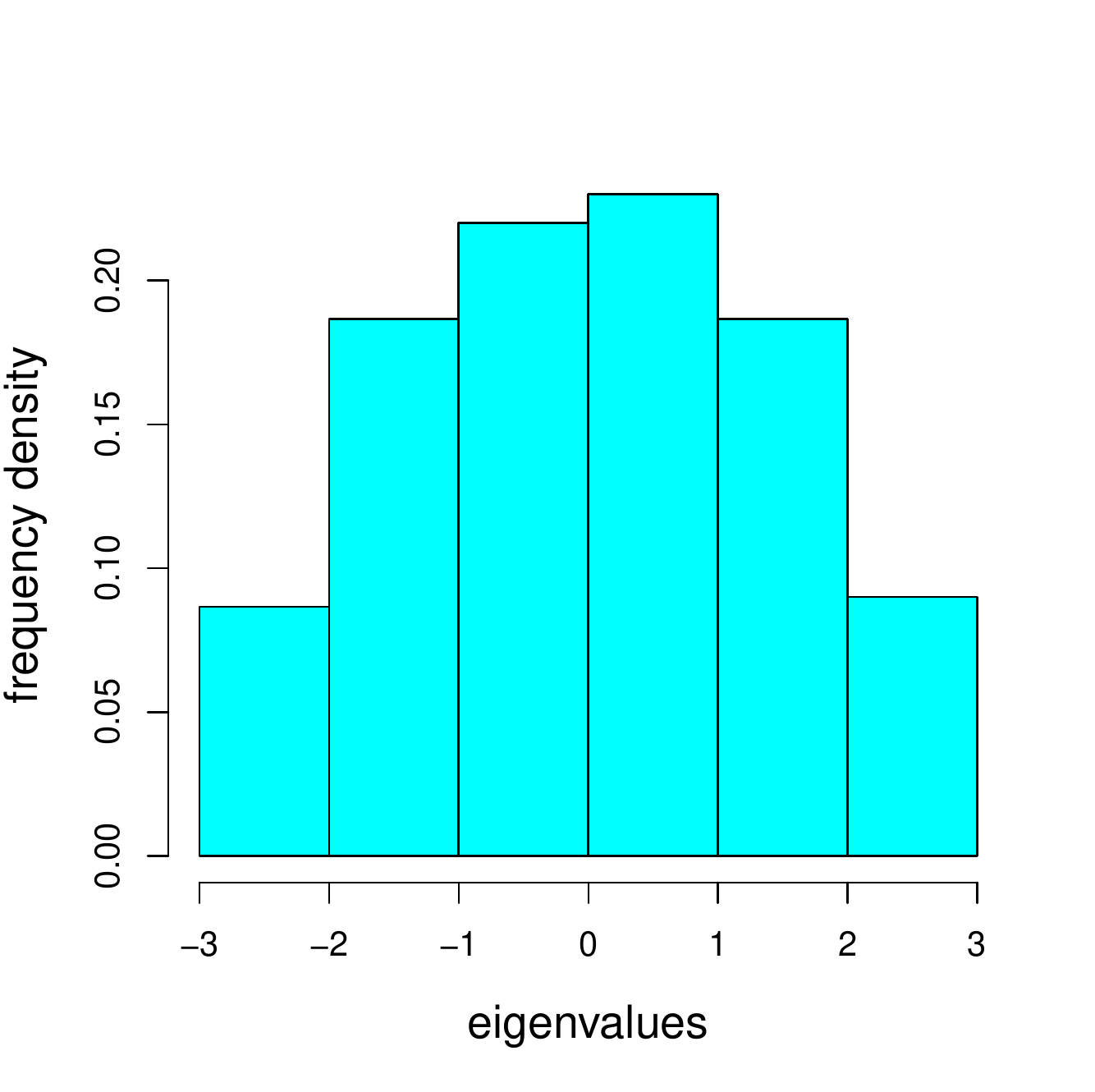}
\includegraphics[height=55mm, width=50mm]{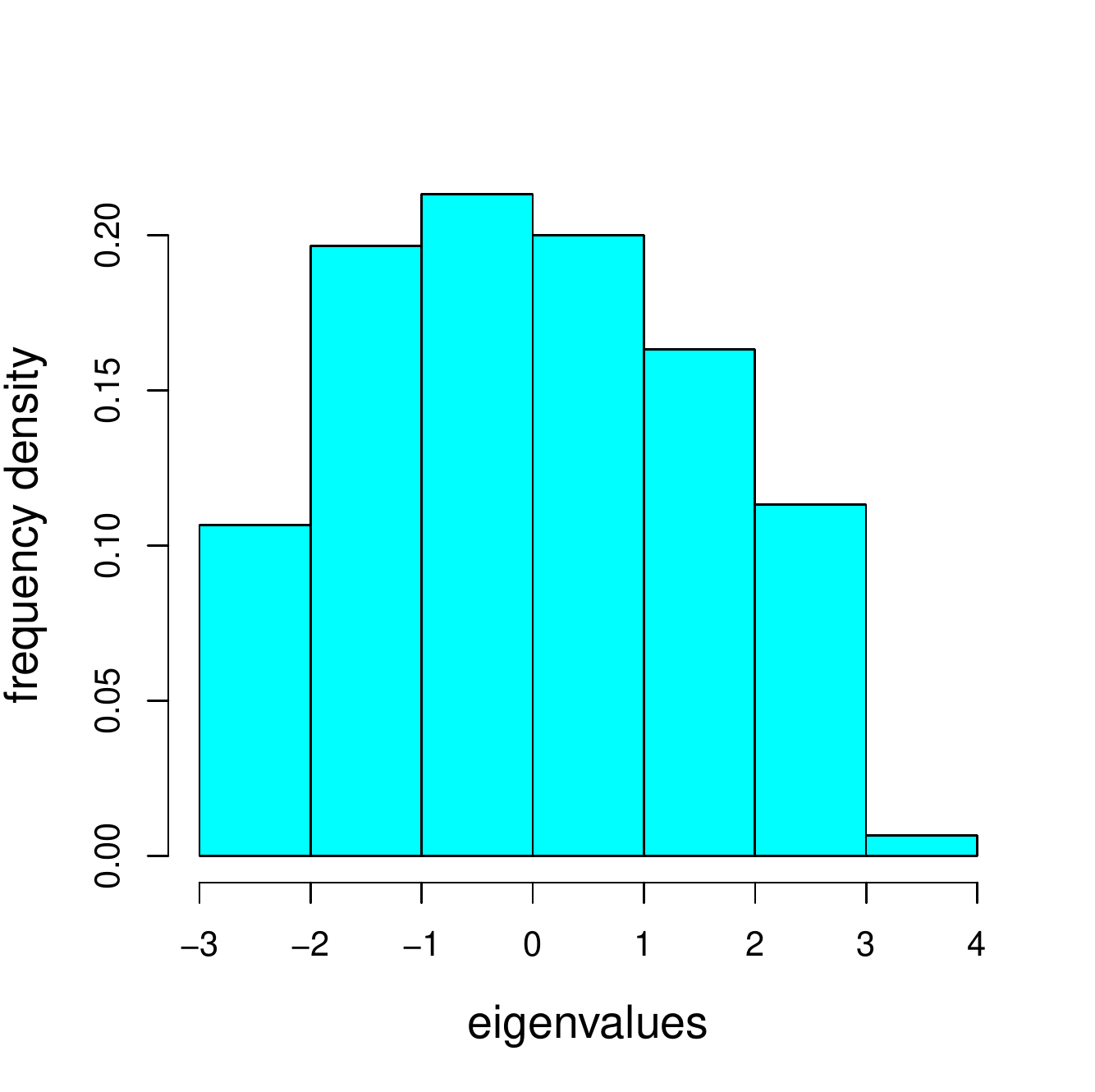}
\includegraphics[height=55mm, width=50mm]{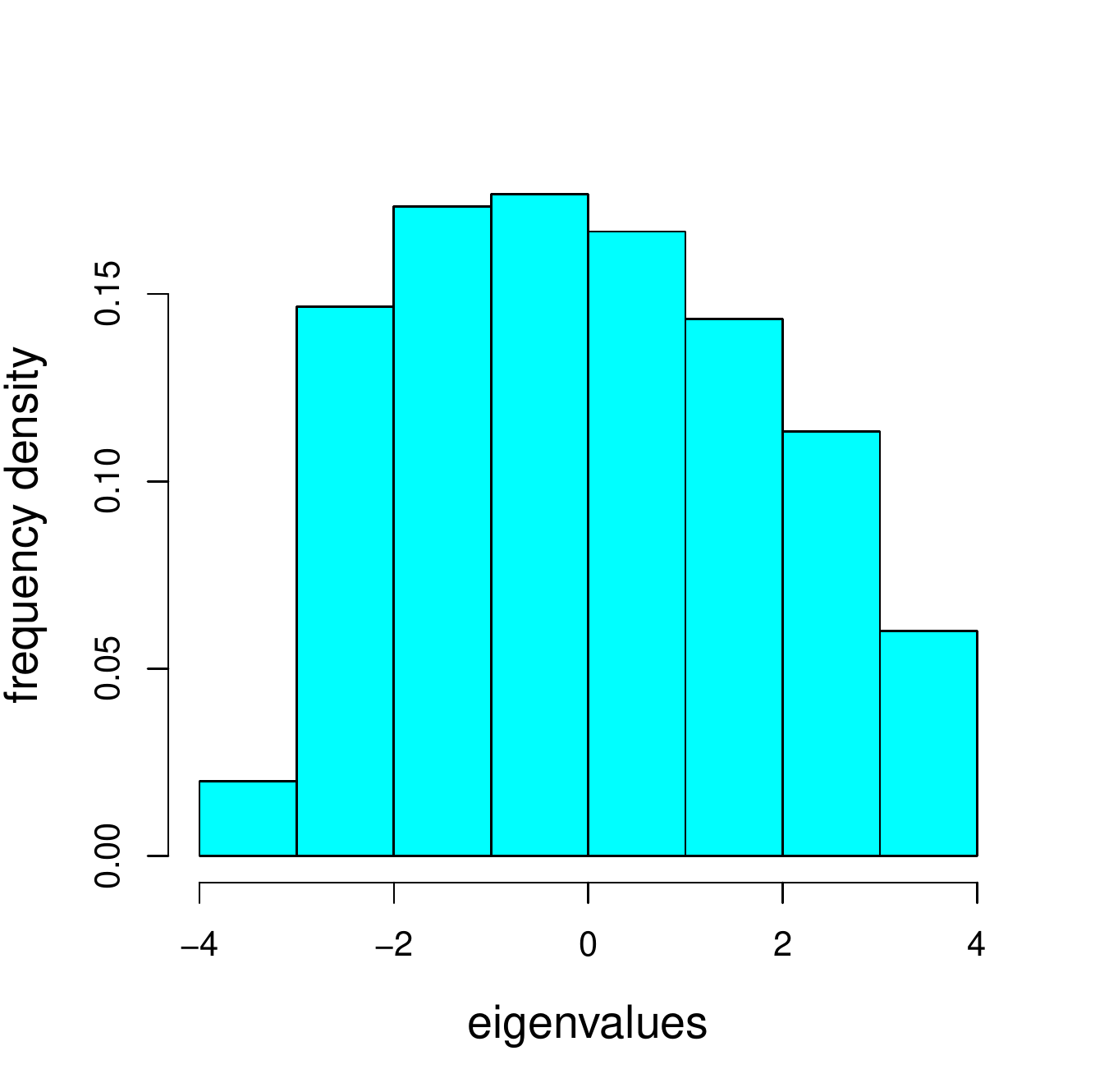}
\includegraphics[height=55mm, width=50mm]{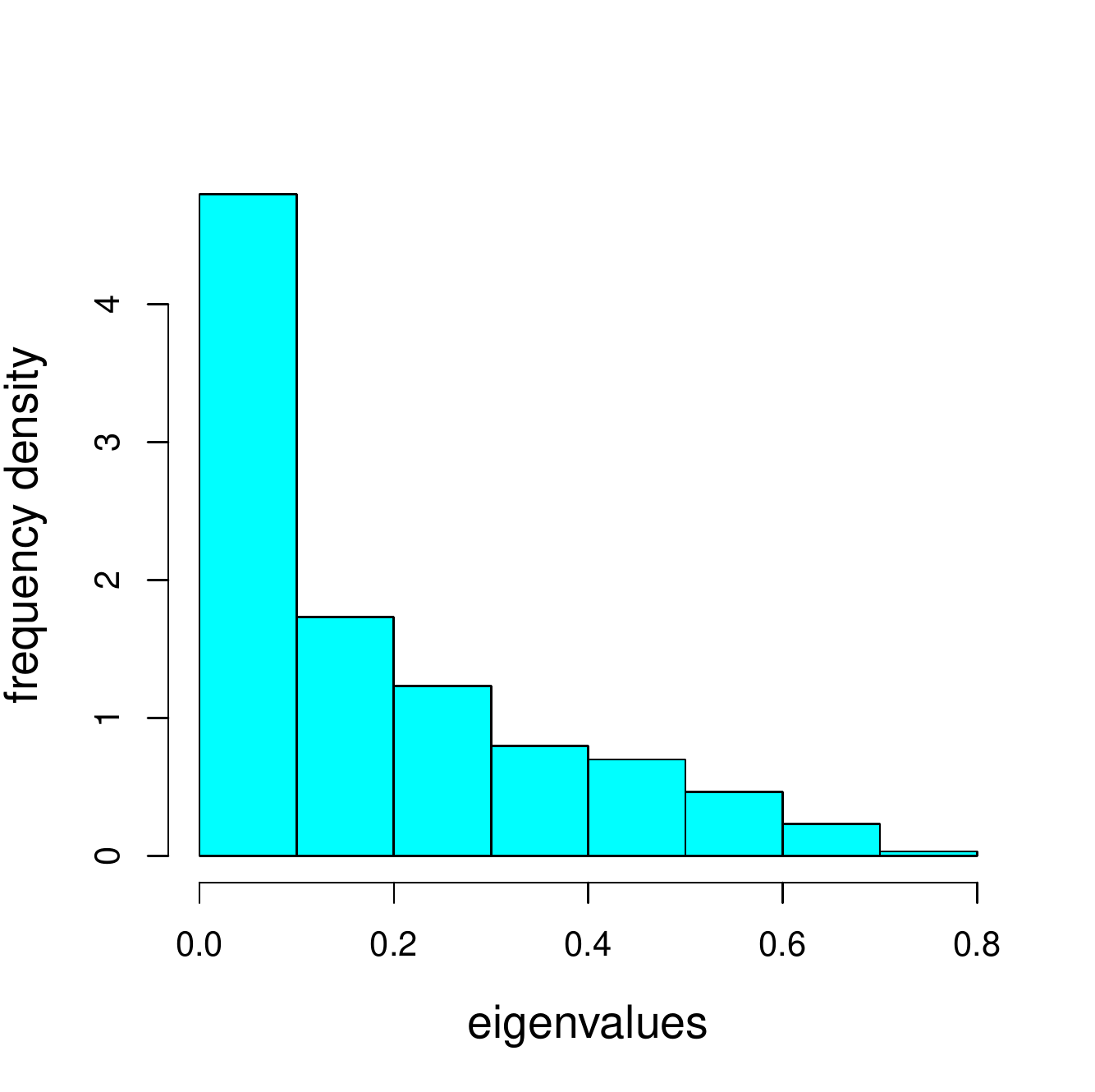}
\includegraphics[height=55mm, width=50mm]{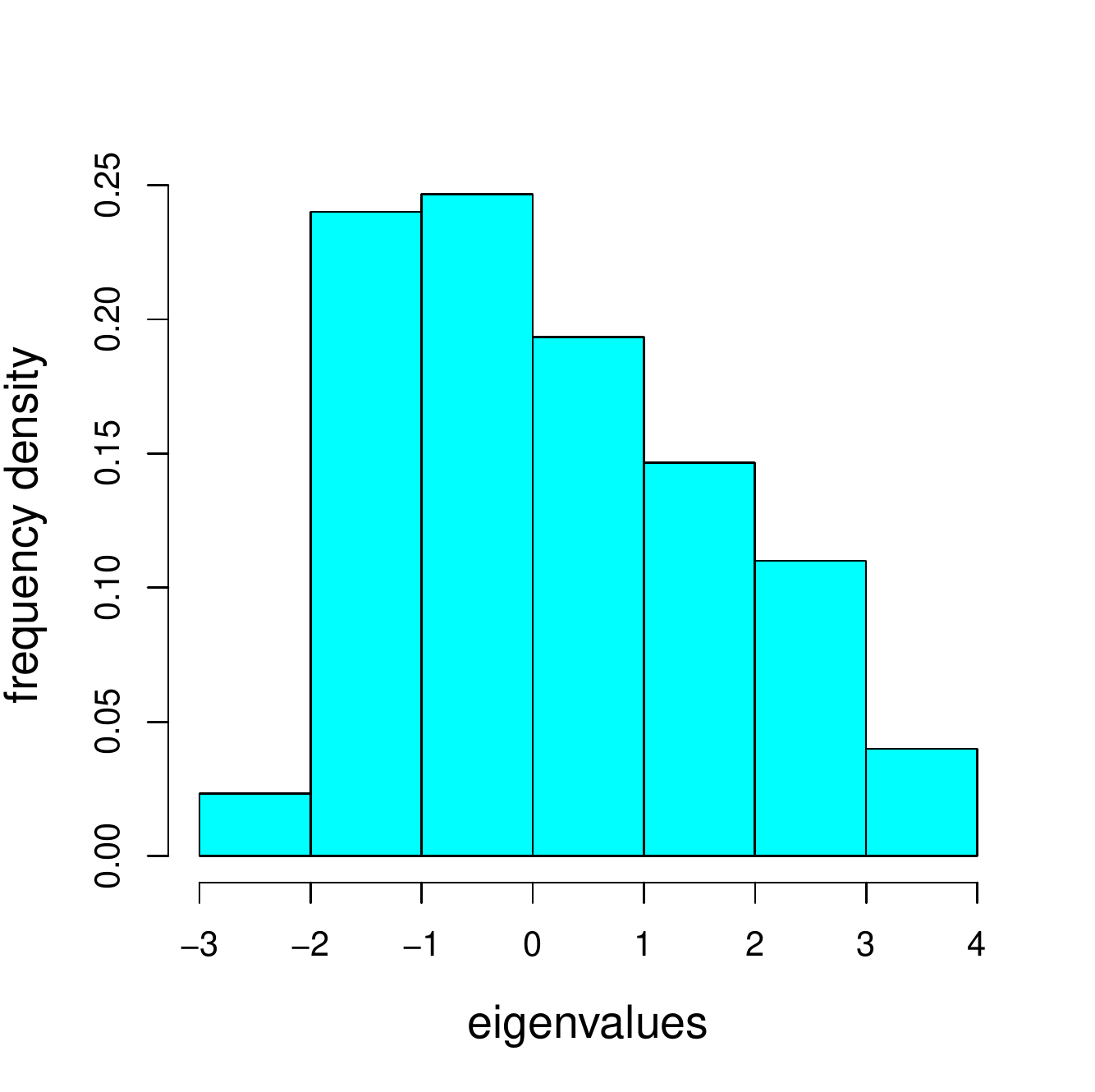}
\includegraphics[height=55mm, width=50mm]{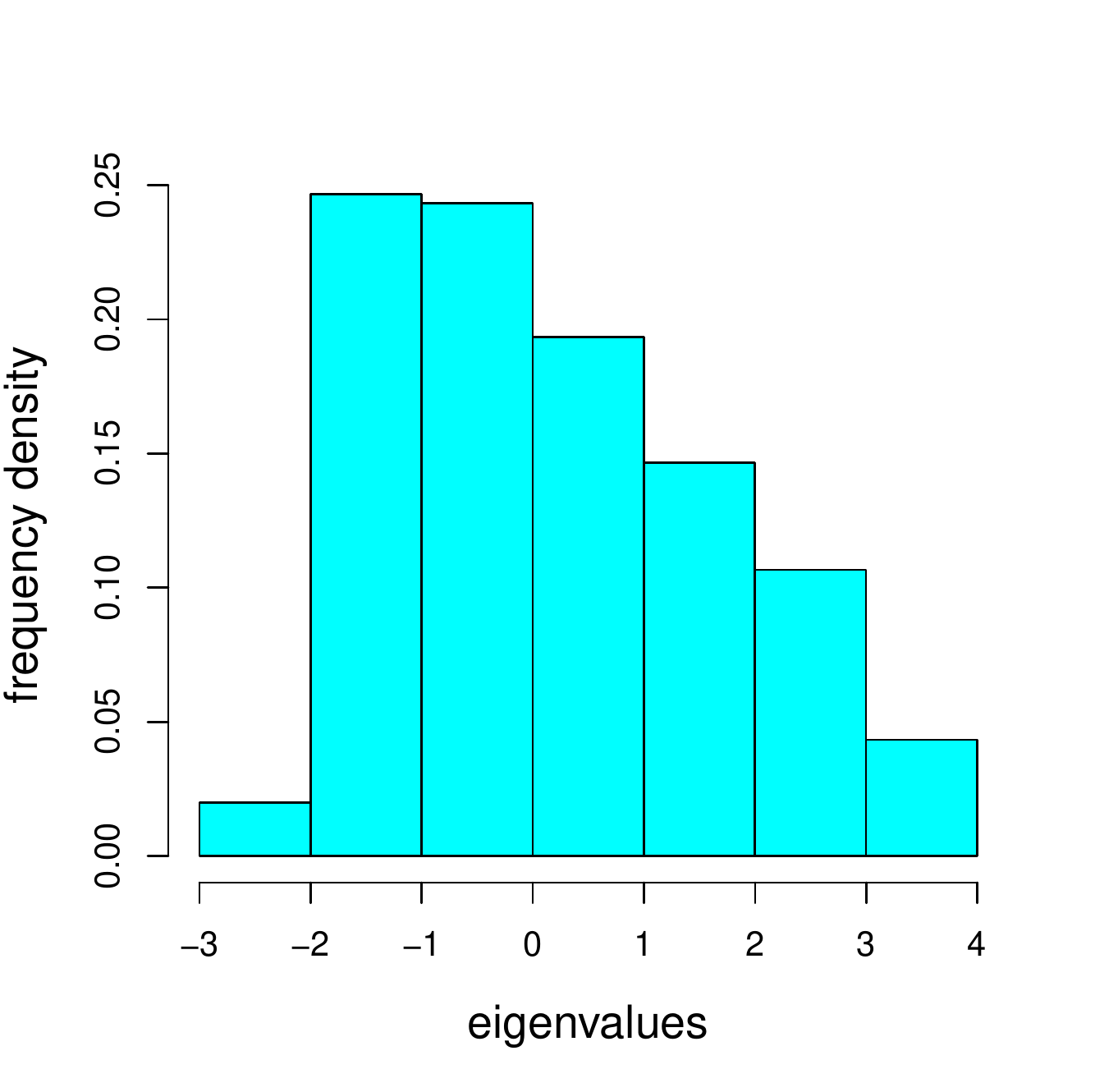}
\includegraphics[height=55mm, width=50mm]{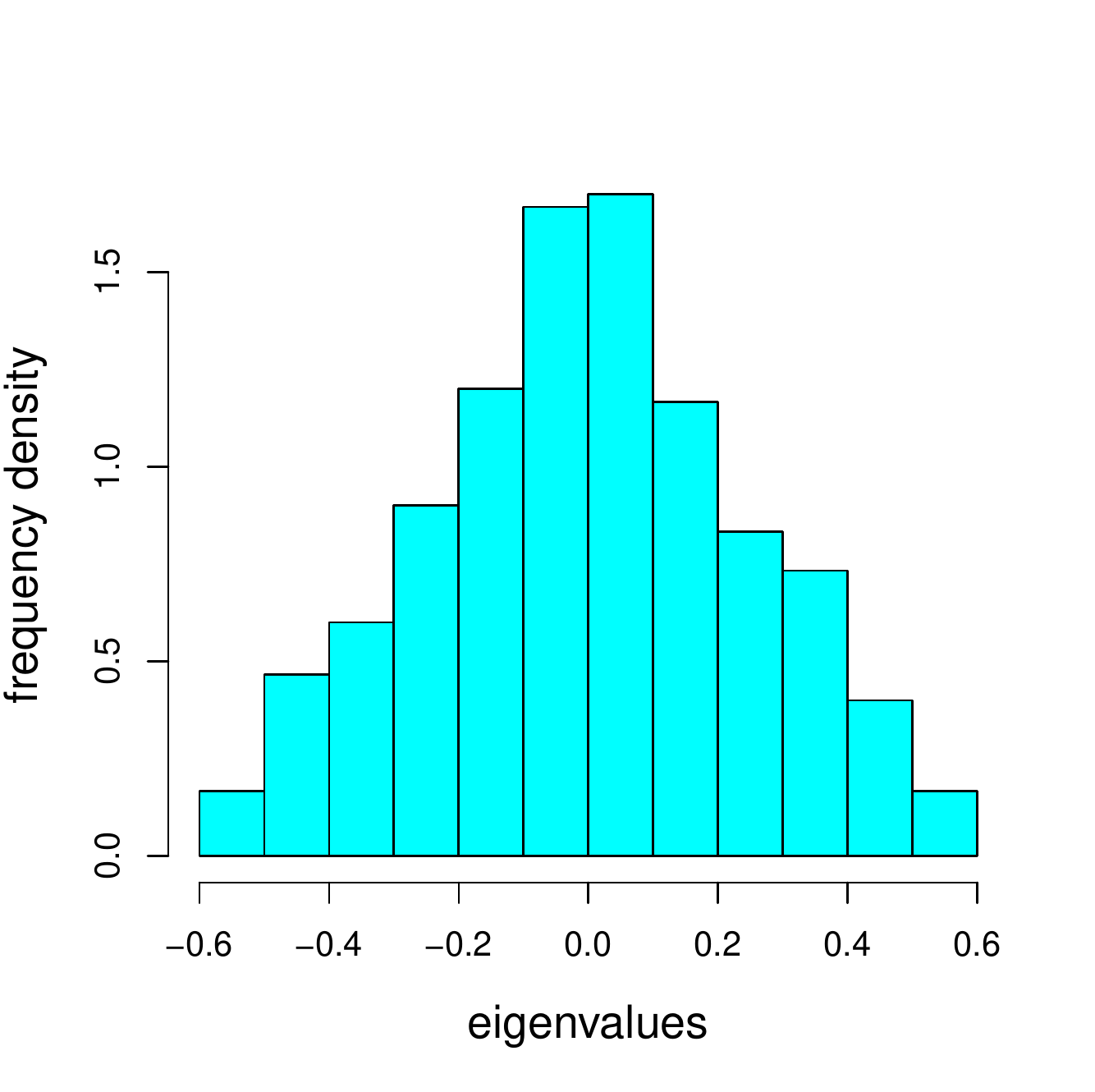}
\includegraphics[height=55mm, width=50mm]{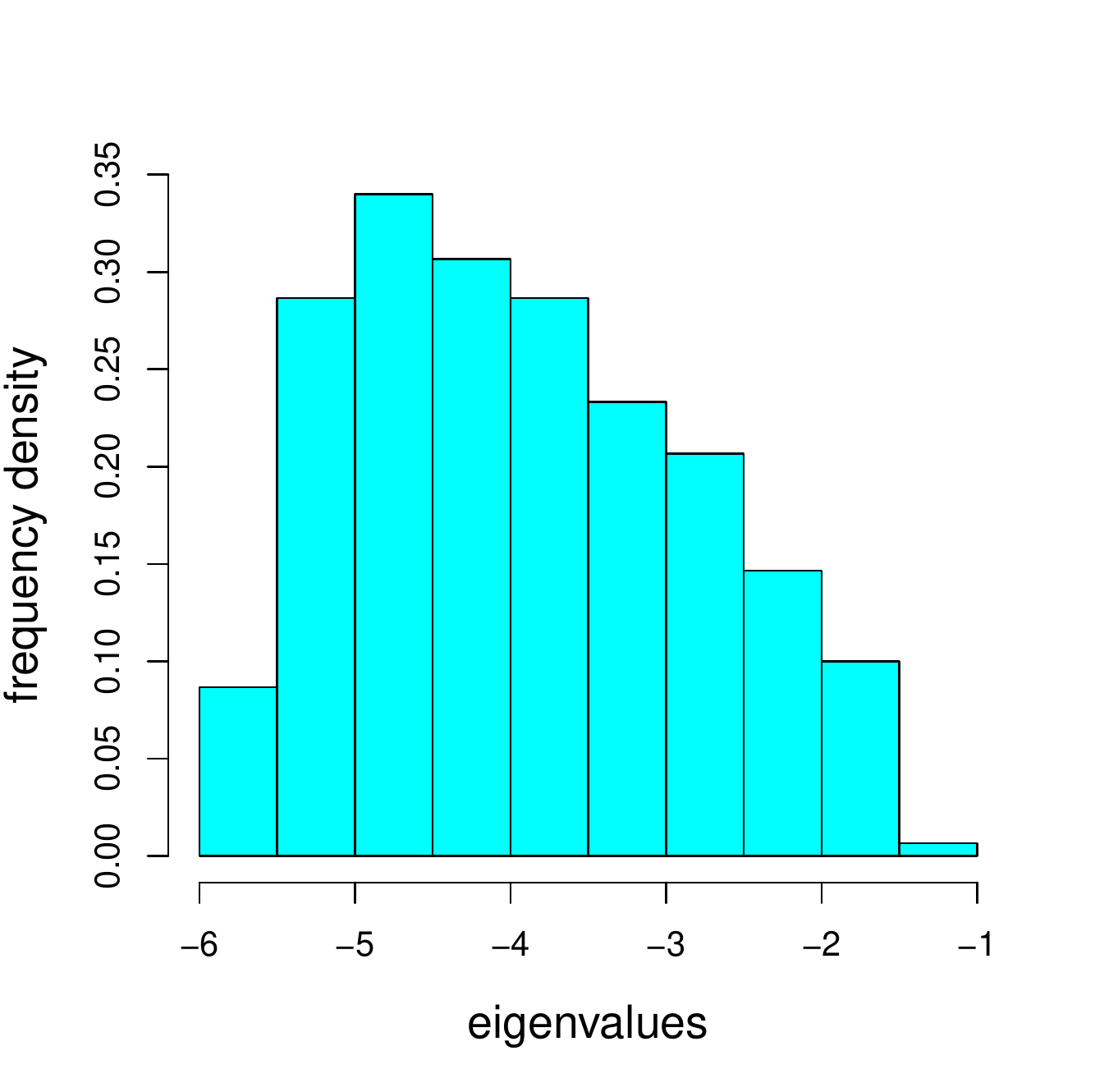}
\includegraphics[height=55mm, width=50mm]{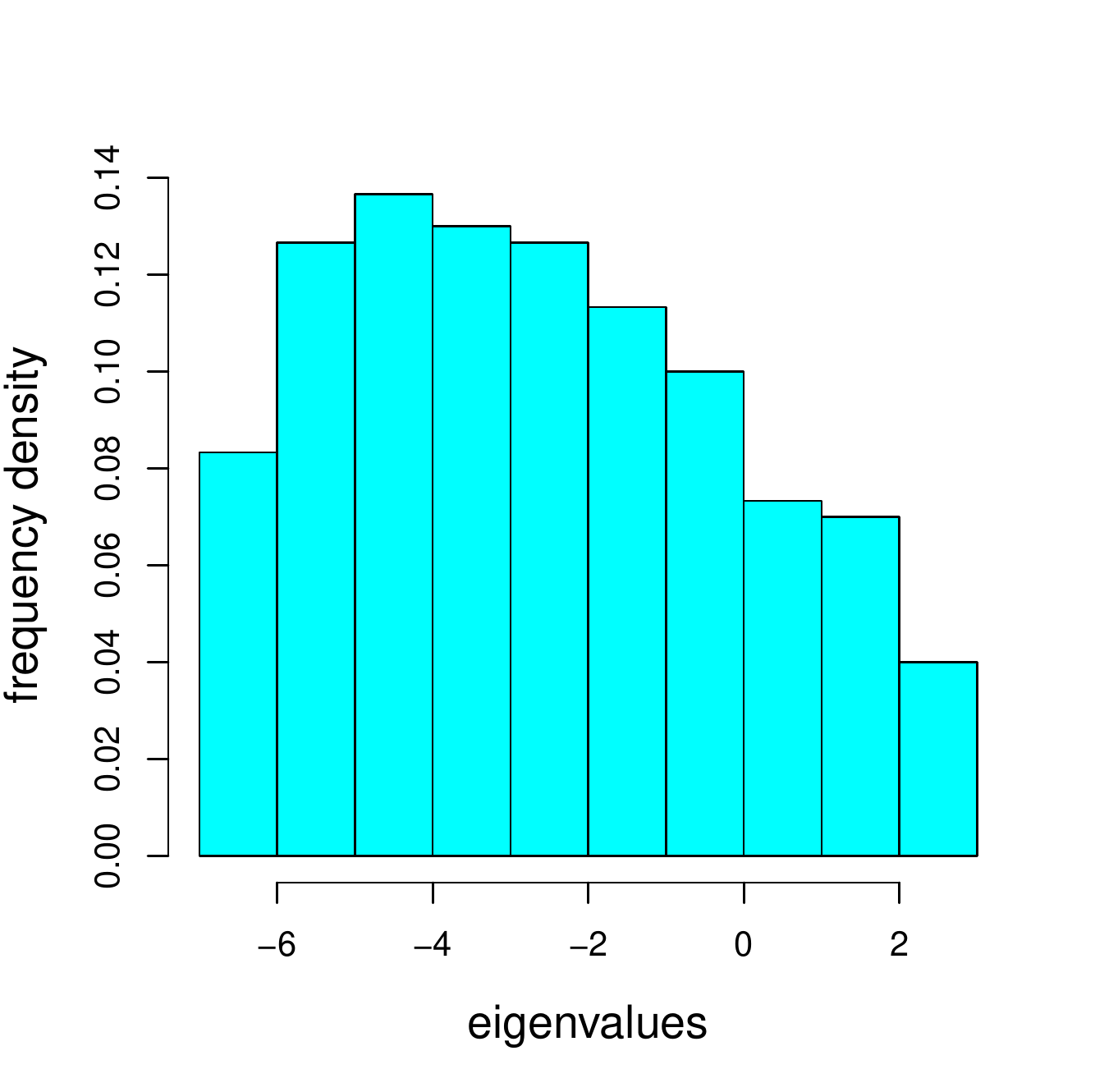}
\includegraphics[height=55mm, width=50mm]{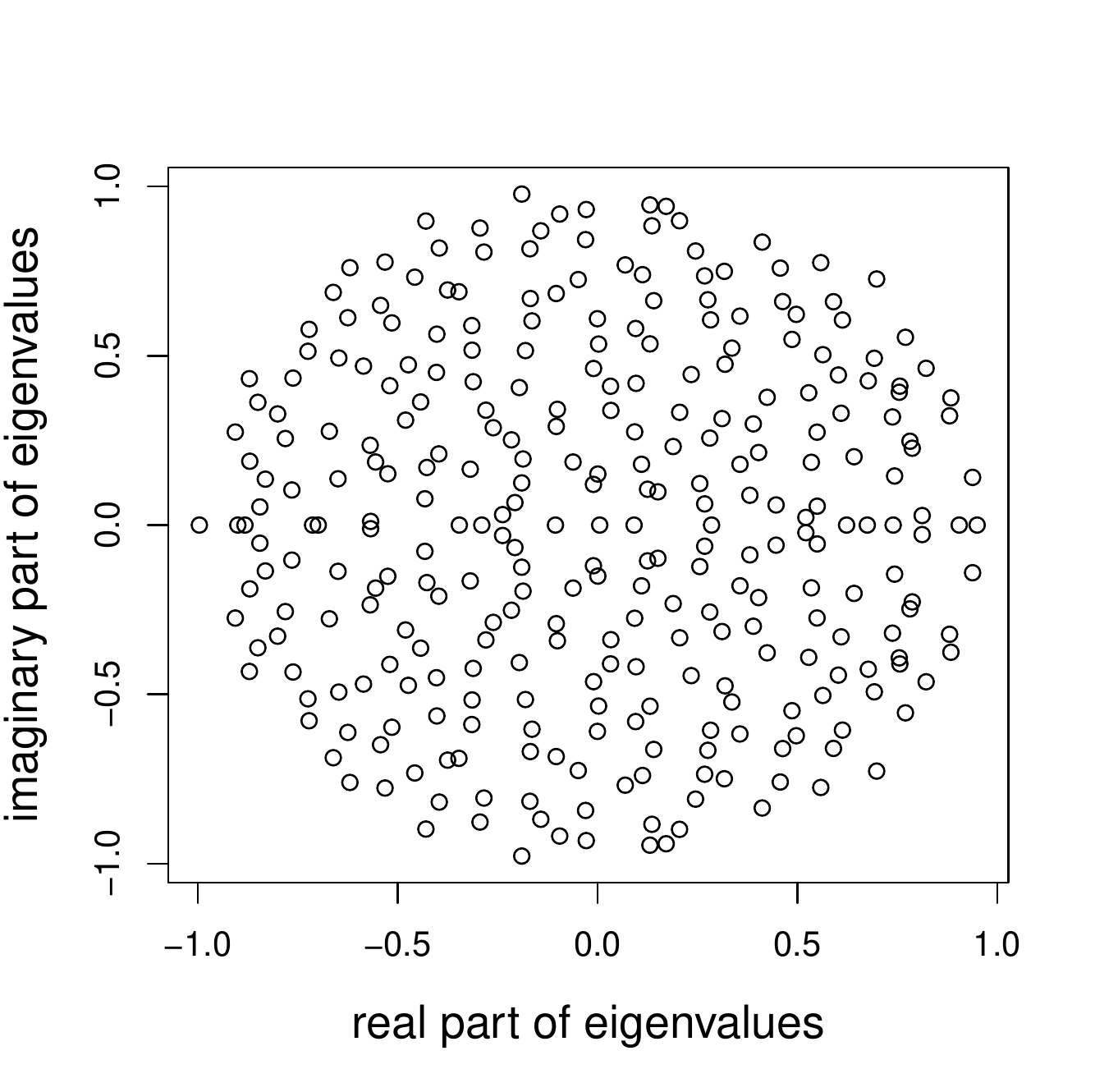}
\includegraphics[height=55mm, width=50mm]{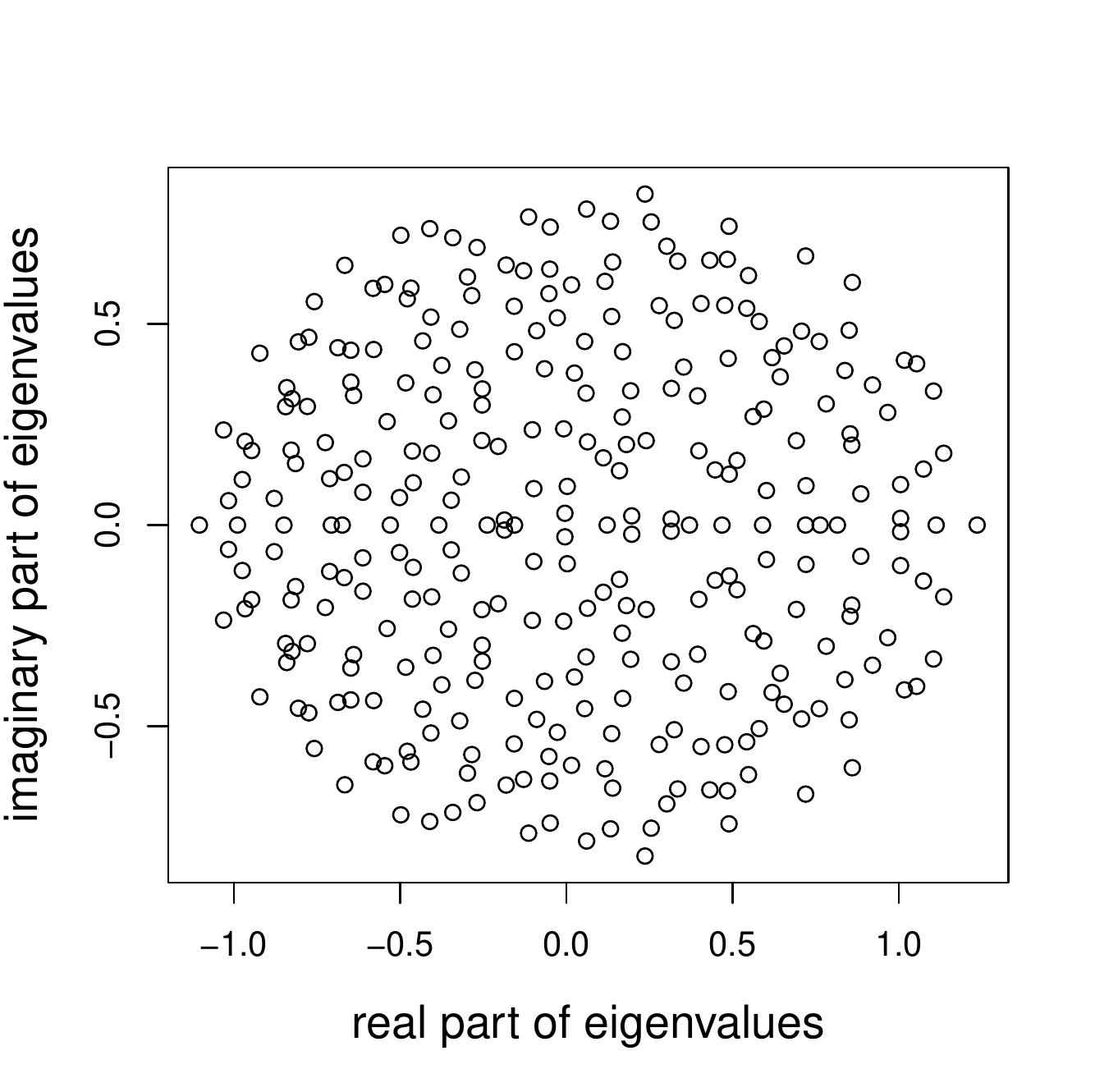}
\includegraphics[height=55mm, width=50mm]{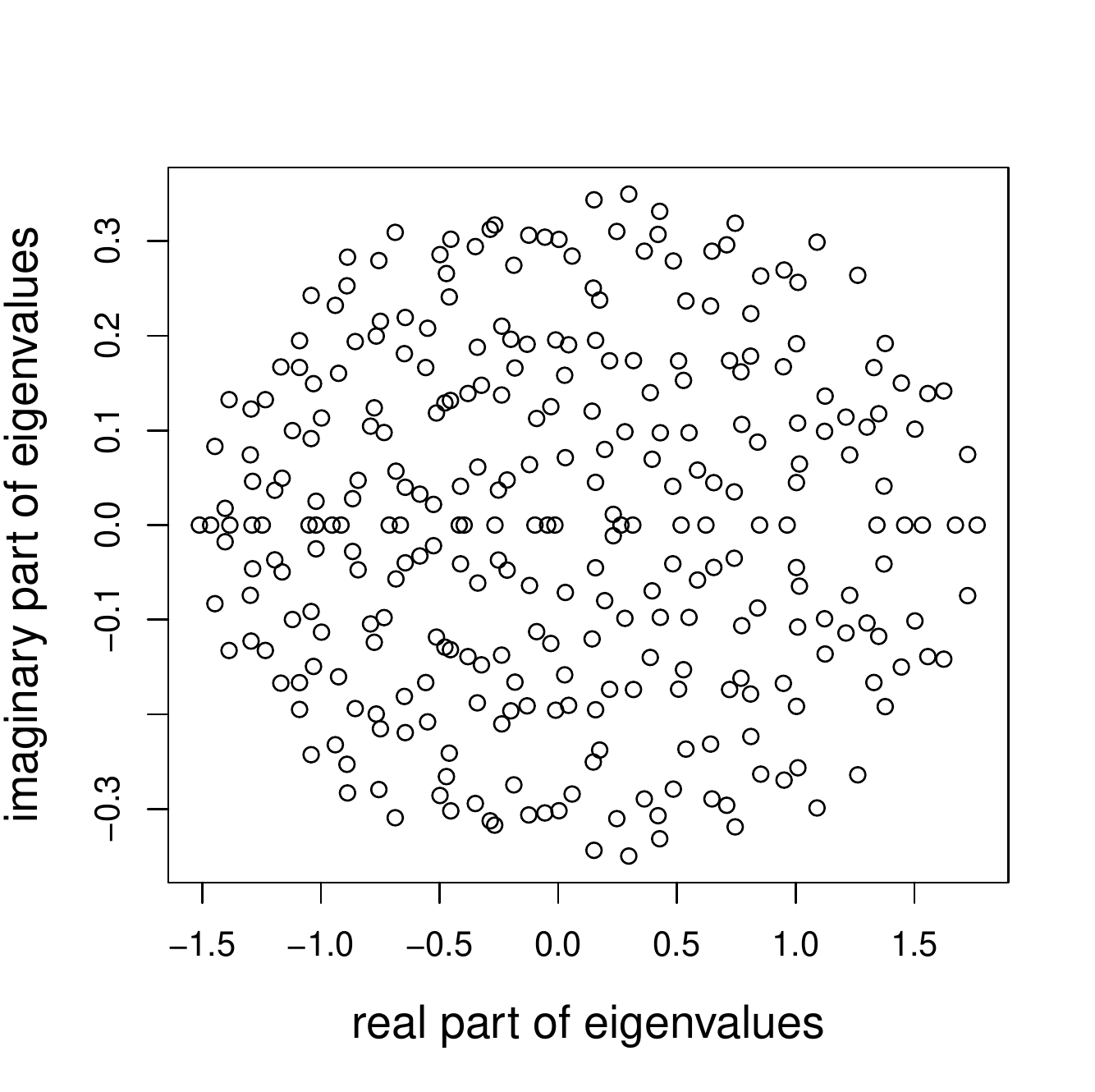}
\caption{Histogram for ESD of $\sqrt{np^{-1}}(C+C^{*}-2\rho\text{I}_p)$, $\sqrt{np^{-1}}(CC^{*}-\rho^2\text{I}_p)$ and $\sqrt{n_2p^{-1}}(C_1C_2^{*}+C_2C_1^{*}- 2\rho^2\text{I}_p)$ ($n_1=n$, $n_2 = 2n$, $\rho_1=\rho_2=\rho$) in Rows 1-3 respectively.  Scatter plot for ESD of $\sqrt{np^{-1}}(C-\rho\text{I}_p)$ in Row 4. All figures are for $n=10000$ and $p=300$. The values of $\rho$ are $0,0.4, 0.8$ respectively in Columns 1-3.} % Column 1: $n=p=500$, $\rho=0$, Column 2: $n=1000, p=500$, $\rho = 0.4$ Column 3: $n=500, p=1000$, $\rho = 0.8$.}
\label{fig2}\end{figure}

\end{document}